\documentclass[]{article}
\usepackage{graphicx}
\usepackage{amsfonts}
\usepackage{amsmath}
\usepackage{amssymb}
\usepackage{fancyhdr}
\usepackage{titlesec}
\usepackage{indentfirst}
\usepackage{booktabs}
\usepackage{verbatim}
\usepackage{color}
\usepackage{amsthm}
\usepackage{hyperref}
\usepackage{subcaption}
\usepackage[page,header]{appendix}
\usepackage{titletoc}
\usepackage{algorithm}
\usepackage{algpseudocode}

\newcommand{\AlgPhase}[1]{%
  \vspace{0.3em} 
  \Statex \textbf{#1} 
  \vspace{0.1em} 
}

\numberwithin{equation}{section}
\newtheorem{theorem}{Theorem}[section]

\newtheorem{proposition}[theorem]{Proposition}

\newtheorem{remark}[theorem]{Remark}
\newtheorem{assumption}{Assumption}[section]

\usepackage{bm}

\newcommand{\dv}{\mathrm{d}v}
\newcommand{\dx}{\mathrm{d}x}

\newcommand{\modifyfirst}[1]{#1}
\newcommand{\modifysecond}[1]{#1}

\makeatletter
\usepackage{fullpage}
\usepackage{url}
\title{A separable and asymptotic-preserving dynamical low-rank method for the Vlasov--Poisson--Fokker--Planck system}
\date{}

\author{Shiheng Zhang\footnote{Department of Applied Mathematics, University of Washington, Seattle, WA 98195 (shzhang3@uw.edu). Corresponding author.} \
	\ and \  Jingwei Hu\footnote{Department of Applied Mathematics, University of Washington, Seattle, WA 98195 (hujw@uw.edu).} }    

\makeatother
\usepackage[english]{babel}
\begin{document}
\maketitle

\begin{abstract}
We present a dynamical low-rank (DLR) method for the Vlasov--Poisson--Fokker--Planck (VPFP) system. Our main contributions are two-fold: (i) a conservative spatial discretization of the 
Fokker--Planck operator that factors into velocity-only and space-only components, enabling efficient low-rank projection, and (ii) a time discretization within the DLR framework that properly handles stiff collisions. We propose both first-order and second-order low-rank IMEX schemes. For the first-order scheme, we prove an asymptotic-preserving (AP) property when the field fluctuation is small. Numerical experiments demonstrate accuracy, robustness, and AP property at modest ranks.
\end{abstract}

{\small 

}

\section{Introduction}
\label{sec:intro}

Kinetic equations provide a mesoscopic  description of dilute particle systems by evolving the phase-space distribution under transport, field effects, and collisions. In this work, we consider the Vlasov--Poisson--Fokker--Planck (VPFP) system
\begin{equation}\label{eq:VPFP-original}
  \partial_t f(x,v,t) + v\cdot\nabla_x f(x,v,t) + \frac{1}{\varepsilon}\,E(x,t)\cdot\nabla_v f(x,v,t)
  \;=\; \frac{1}{\varepsilon}\,\left(\Delta_v f(x,v,t) + \nabla_v \cdot (v f(x,v,t))\right),
\end{equation}
where $f(x,v,t)$ is the distribution function depending on time $t$, position \(x\in\Omega_x\subset\mathbb{R}^{d_x}\), and velocity \(v\in\mathbb{R}^{d_v}\). $E(x,t)$ is the electric field governed by the Poisson equation with potential \(\phi(x,t)\):
\begin{equation}\label{eq:Poisson}
  E(x,t) = -\nabla_x\phi(x,t),\qquad -\Delta_x\phi(x,t)=\rho(x,t)-\eta(x),
\end{equation}
where \(\rho(x,t)=\int_{\mathbb{R}^{d_v}}f\,\dv\) is the charge density and $\eta(x)$ is a given background charge satisfying the global neutrality condition \(\int_{\Omega_x}(\rho(x,0)-\eta(x))\,\dx=0\). The collision operator on the right-hand side of \eqref{eq:VPFP-original} is the classical linear Fokker--Planck operator \cite{dougherty1964model, villani2002review}. \(\varepsilon>0\) is a nondimensional parameter determining the strength of both the electric field and collisions. \modifyfirst{The regime \(\varepsilon\to 0\) is characterized by rapid relaxation in velocity and macroscopic high-field dynamics in physical space \cite{nieto2001high,poupaud1992runaway, cercignani2000device}. In this paper, we focus on the high-field scaling, which leads to a hyperbolic equation in the limit, as shown below. We also note that the alternative low-field regime, in which the field term is $\mathcal{O}(\varepsilon)$ smaller than the Fokker--Planck operator, instead leads to a parabolic limit~\cite{poupaud2000parabolic, arnold2001low}.} 

Since \(E(x,t)\) is independent of \(v\), we have \(\nabla_v\cdot(E f)=E\cdot\nabla_v f\). We may absorb the high-field drift into the right-hand side of \eqref{eq:VPFP-original} to obtain a more succinct form
\begin{equation}\label{eq:VPFP}
     \partial_t f + v\cdot\nabla_x f \;=\; \frac{1}{\varepsilon}\,\nabla_v\cdot\left(M\,\nabla_v\left(\tfrac{f}{M}\right)\right),
\end{equation}
where the local Maxwellian \(M\) is defined as
\begin{equation}\label{eq:Maxwellian_def}
   M(x,v,t)=\frac{1}{(2\pi)^{d_v/2}}\exp\left(-\tfrac12\,|v-E(x,t)|^2\right).
\end{equation}
Taking the zeroth and first velocity moments of \eqref{eq:VPFP} (i.e., $\int \cdot\,\dv$ and $\int \cdot\, v\,\dv$) yields the continuity and momentum balances
\begin{align}
  \partial_t \rho + \nabla_x \cdot J &= 0, \label{eq:continuity}\\
  \partial_t J + \nabla_x \cdot \int_{\mathbb{R}^{d_v}}v\otimes v f\,\dv &= \frac{1}{\varepsilon}\,(\rho E - J), \label{eq:momentum}
\end{align}
where \(J(x,t)=\int fv\,\dv\) is the current density. In the high-field limit, the stiff relaxation on the right-hand side of \eqref{eq:momentum} drives \(J \approx \rho E\) to leading order. Substituting this into \eqref{eq:continuity} yields the closed macroscopic system
\begin{equation}\label{eq:macro_drift}
  \begin{cases}
  \partial_t \rho + \nabla_x \cdot (\rho E) = 0, \\[2pt]
  E = -\nabla_x \phi, \quad -\Delta_x \phi = \rho - \eta.
  \end{cases}
\end{equation}
This macroscopic drift system represents the high-field limit of the VPFP system as \(\varepsilon \to 0\).

Numerical simulation of the VPFP system presents two principal difficulties: high dimensionality and stiffness. The phase-space distribution \(f(x,v,t)\) evolves in $\mathbb{R}^{d_x+d_v}$. Discretizing this domain leads to computational and storage costs on the order of \(\mathcal{O}(N_x^{d_x} N_v^{d_v})\), which becomes prohibitive in high dimensions. To mitigate this ``curse of dimensionality'', a promising approach is the dynamical low-rank (DLR) approximation \cite{koch2007dynamical, lubich2014projector}. 
This approach constrains the solution to a manifold of fixed rank \(r\):
\[
f(x,v,t) \approx \sum_{i,j=1}^r X_i(x,t)\, S_{ij}(t)\, V_j(v,t),
\]
where the orthonormal factors \(\{X_i\}_{i=1}^{r}\) and \(\{V_j\}_{j=1}^{r}\) are evolved by projecting the governing equation onto the tangent space of the low-rank manifold. 
The efficiency of the DLR approximation relies critically on the structure of the operators in the governing equation. When the operator exhibits a separable structure in $x$ and $v$, the computational cost scales as $\mathcal{O}(r^2(N_x^{d_x} + N_v^{d_v}) + r^4)$, and the storage requirement scales as $\mathcal{O}(r(N_x^{d_x} + N_v^{d_v}) + r^2)$. This substantial saving in the regime $r \ll \min\{N_x^{d_x}, N_v^{d_v}\}$ is what makes high-dimensional kinetic simulations feasible. In contrast, if the operator lacks a separable structure, evaluating the dynamics typically requires reconstructing the full tensor, which reverts the cost to the prohibitive $\mathcal{O}(N_x^{d_x} N_v^{d_v})$. Recently developed interpolatory DLR methods \cite{carrel2025interpolatory, dektor2025interpolatory} bypass this issue by employing collocation or empirical interpolation. However, additional error is introduced by replacing the orthogonal projection in the DLR method with an oblique projection. There are many other recent developments on DLR methods applied to general kinetic equations and we refer the reader to the review paper \cite{EKKMQ25}.

Due to the stiff field term and collision term, applying the DLR method to the VPFP system in the high-field regime is not straightforward. Explicit time integration requires restrictive time steps \(\Delta t=\mathcal{O}(\varepsilon)\). This necessitates asymptotic-preserving (AP) schemes, which allow \(\varepsilon\)-independent time steps while recovering the correct macroscopic limit \cite{jin1999efficient, jin2022asymptotic}. While AP schemes for the VPFP system are well-established for full-tensor methods \cite{jin2011asymptotic, carrillo2021variational, zhu2017hypocoercivity}, extending them to the DLR framework is non-trivial. \modifysecond{Recent works have developed AP DLR formulations for several kinetic models, including linear transport  equations in the diffusion limit~\cite{ding2021dynamical,einkemmer2021asymptotic,einkemmer2024asymptotic}, the nonlinear Boltzmann equation in the fluid limit~\cite{einkemmer2025asymptotic}, and the radiative transfer equation \cite{FKP25}. In these models, the collision operator is local in space and acts only on the velocity (or angular) variable, so its action on a low-rank ansatz preserves the separable structure and admits an efficient low-rank projection.} The central difficulty in the present setting lies in the Fokker--Planck operator on the right-hand side of \eqref{eq:VPFP}: the spatially varying field \(E(x,t)\) creates a strong coupling between \(x\) and \(v\) that destroys the separability required for efficient low-rank projection\footnote{The separability is not an issue if one uses the original formulation \eqref{eq:VPFP-original}. However, this form makes the design of an AP scheme difficult.}.
This non-separability issue was addressed in \cite{coughlin2022efficient} by evolving the quotient $f/M$ rather than $f$, followed by a complicated change of variables. In this paper, we propose a direct approach: a conservative and separable discretization of the Fokker--Planck operator that allows efficient low-rank projection without variable transformation or full tensor reconstruction. Furthermore, the stiff Fokker--Planck operator in the low-rank framework requires careful time discretization to ensure stability. We propose both first-order and second-order low-rank IMEX schemes. For the first-order scheme, we
prove an AP property when the field fluctuation is small. 


The rest of this paper is organized as follows. 
Section~\ref{sec:dlr_algorithm} establishes the dynamical low-rank formulation for the VPFP system.
Section~\ref{sec:discretization} describes the separable, conservative discretization of the Fokker--Planck operator and its low-rank projection.
Section~\ref{sec:time_discretization} presents the fully discrete first-order and second-order low-rank IMEX schemes.
Section~\ref{sec:AP} provides the asymptotic-preserving analysis for the first-order low-rank scheme.
Section~\ref{sec:numerical_experiments} reports numerical experiments demonstrating the accuracy and robustness of the proposed methods, and we conclude in Section~\ref{sec:con}.

\section{Dynamical low-rank formulation for the VPFP system}
\label{sec:dlr_algorithm}

In this section, we describe the projector-splitting dynamical low-rank (DLR) method, originally proposed by Lubich and Oseledets \cite{lubich2014projector}, applied to the VPFP system  adopting the form \eqref{eq:VPFP}. We define the advection operator $\mathcal{A}$ and the Fokker--Planck operator $\mathcal{L}$ as
\begin{align}
  \mathcal{A}(f) &\,:=\; -\,v\cdot\nabla_x f, \label{eq:op_A_def} \\
  \mathcal{L}(f) &\,:=\; \nabla_v \cdot \left(M\,\nabla_v \left(\tfrac{f}{M}\right)\right),\label{eq:op_L_def}
\end{align}
so that  \eqref{eq:VPFP} can be written as 
\begin{equation}
\label{eq:VPFP-1}
    \partial_t f = \mathcal{A}(f) + \frac{1}{\varepsilon}\mathcal{L}(f).
\end{equation}

The DLR method approximates the solution \(f(x,v,t)\) by constraining it to the manifold \(\mathcal{M}_r\) of rank-\(r\) functions. We employ the low-rank ansatz
\begin{equation}\label{eq:rank_dlra_sum}
  f(x,v,t) \;=\;
\sum_{i,j=1}^r X_i(x,t)\,S_{ij}(t)\,V_j(v,t),
\end{equation}
where \(S(t) \in \mathbb{R}^{r \times r}\) is the coefficient matrix, and \(\{X_i(x,t)\}_{i=1}^r \subset L^2(\Omega_x)\) and \(\{V_j(v,t)\}_{j=1}^r \subset L^2(\mathbb{R}^{d_v})\) are orthonormal bases satisfying
\[
  \langle X_i(\cdot,t),X_{i'}(\cdot,t)\rangle_x=\delta_{ii'},
  \qquad
  \langle V_j(\cdot,t),V_{j'}(\cdot,t)\rangle_v=\delta_{jj'}.
\]
Here $\langle \phi, \psi\rangle_x:=\int_{\Omega_x} \phi(x)\psi(x)\,\dx$ and $\langle \phi, \psi\rangle_v:=\int_{\mathbb{R}^{d_v}} \phi(v)\psi(v)\,\dv$ denote the spatial and velocity inner products, respectively; we will also need $\langle \phi, \psi\rangle_{xv}:=\int_{\Omega_x}\int_{\mathbb{R}^{d_v}}  \phi(x,v)\psi(x,v)\,\dx\,\dv$ for the full phase-space inner product. 

According to the Dirac--Frenkel variational principle \cite{koch2007dynamical, lubich2008quantum}, the evolution of the low-rank approximation is determined by the orthogonal projection of the vector field \(\mathcal{A}(f) + \frac{1}{\varepsilon}\mathcal{L}(f)\) onto the tangent space \(T_f\mathcal{M}_r\) at the current approximation \(f\):
\begin{equation}\label{eq:projected_eq_dlr_sum}
  \partial_t f \;=\; P_f\big(\mathcal{A}(f) + \frac{1}{\varepsilon}\mathcal{L}(f)\big).
\end{equation}
The projection operator \(P_f: L^2(\Omega_x \times \mathbb{R}^{d_v}) \to T_f\mathcal{M}_r\) is given explicitly by \cite{lubich2014projector}:
\begin{equation}\label{eq:Pf_standard_sum}
  P_f(g) \;=\; (I \otimes P_V)\,g \;-\; (P_X \otimes P_V)\,g \;+\;
(P_X \otimes I)\,g,
\end{equation}
where \(P_X\) and \(P_V\) are the orthogonal projectors onto the subspaces spanned by \(\{X_i\}_{i=1}^{r}\) and \(\{V_j\}_{j=1}^{r}\), respectively.

Direct integration of \eqref{eq:projected_eq_dlr_sum} is numerically challenging due to the curvature of the manifold. We employ the robust projector-splitting integrator \cite{lubich2014projector}, which splits the projection \eqref{eq:Pf_standard_sum} into three coupled differential equations (K--S--L splitting) solved sequentially.

\paragraph{K-step (update the spatial basis).}
We first solve the subproblem corresponding to the term \((I \otimes P_V)(\mathcal{A}+\frac{1}{\varepsilon}\mathcal{L})(f)\). In this step, the velocity basis \(\{V_j\}_{j=1}^{r}\) is frozen. We represent the solution as
\begin{equation}\label{eq:f_K_form}
    f(x,v,t) = \sum_{j=1}^r K_j(x,t)\,V_j(v,t), \qquad
    K_j(x,t) = \sum_{i=1}^r X_i(x,t)\,S_{ij}(t).
\end{equation}
Projecting the dynamics onto the fixed basis \(V_j\) yields a system of \(r\) evolution equations for the spatial coefficients \(K_j\):
\begin{equation}\label{eq:K_comp_evo}
    \partial_t K_j(x,t) = \left\langle (\mathcal{A} + \frac{1}{\varepsilon}\mathcal{L}) \left(\sum_{\ell=1}^r K_\ell\,V_\ell\right),\, V_j \right\rangle_v,
    \quad j=1,\dots,r.
\end{equation}
After evolving \eqref{eq:K_comp_evo} for one time step, a QR factorization of \(\{K_j(\cdot, t^{n+1})\}_{j=1}^{r}\) yields the orthonormal spatial basis \(\{X_i(x,t^{n+1})\}_{i=1}^{r}\) and an intermediate coefficient matrix \(S^{(1)}\).

\paragraph{S-step (update the coefficient matrix).}
Next, we solve the subproblem corresponding to \(-(P_X \otimes P_V)(\mathcal{A}+\frac{1}{\varepsilon}\mathcal{L})(f)\). Here, both bases are frozen. The evolution is restricted to the coefficient matrix \(S\):
\begin{equation}\label{eq:S_comp_evo}
    \frac{d S_{ij}}{dt}
    \;=\;
-\,\left\langle (\mathcal{A} + \frac{1}{\varepsilon}\mathcal{L}) \left(\sum_{k,\ell=1}^r X_k\,S_{k\ell}\,V_{\ell}\right),\; X_i\,V_j \right\rangle_{xv},
    \quad i,j=1,\dots,r.
\end{equation}
This step updates the coefficient matrix from \(S^{(1)}\) to \(S^{(2)}\). The negative sign implies this step effectively evolves backward in time.

\paragraph{L-step (update the velocity basis).}
Finally, we solve the subproblem corresponding to \((P_X \otimes I)(\mathcal{A}+\frac{1}{\varepsilon}\mathcal{L})(f)\). In this step, the spatial basis \(\{X_i\}_{i=1}^{r}\) is frozen. We  represent the solution as
\begin{equation}\label{eq:f_L_form}
    f(x,v,t) = \sum_{i=1}^r X_i(x,t)\,L_i(v,t), \qquad
    L_i(v,t) = \sum_{j=1}^r S_{ij}(t)\,V_j(v,t).
\end{equation}
Projecting onto the fixed basis \(X_i\) yields \(r\) evolution equations for the velocity coefficients \(L_i\):
\begin{equation}\label{eq:L_comp_evo}
    \partial_t L_i(v,t) = \left\langle (\mathcal{A} + \frac{1}{\varepsilon}\mathcal{L}) \left(\sum_{k=1}^r X_k\,L_k\right),\, X_i \right\rangle_x,
    \quad i=1,\dots,r.
\end{equation}
After evolving \eqref{eq:L_comp_evo} for one time step, a QR factorization of \(\{L_i(\cdot, t^{n+1})\}_{i=1}^{r}\) yields the orthonormal velocity basis \(\{V_j(v,t^{n+1})\}_{j=1}^{r}\) and the final coefficient matrix \(S(t^{n+1})\).

\medskip
\noindent\textbf{Separability of the advection operator.}
The efficiency of the DLR method relies on the separability of the operator. The advection operator $\mathcal{A}(f) = -v \cdot \nabla_x f$ naturally satisfies this requirement. Its projection onto the low-rank basis does not require reconstructing the full distribution $f$. For example, in the K-step \eqref{eq:K_comp_evo}, the projection onto the velocity basis $V_j$ becomes:
\[
  \langle \mathcal{A}(f), V_j \rangle_v
  = -\sum_{\ell=1}^r \langle vV_\ell,V_j\rangle_v\cdot \nabla_x K_\ell,
  \]
where the velocity and space dependent terms can be computed separately. If one chooses to discretize the differential operator first and then perform the projection, standard linear finite-difference approximations (e.g., upwind scheme) will preserve this separable structure.

\medskip
\noindent\textbf{The non-separable Fokker--Planck bottleneck.}
In contrast, the Fokker--Planck operator $\mathcal{L}$ in \eqref{eq:op_L_def} poses a significant challenge, where the Maxwellian function $M(x,v,t)$, as defined in \eqref{eq:Maxwellian_def}, depends on the macroscopic electric field $E(x,t)$, introducing a non-trivial coupling between $x$ and $v$. Consequently, its projection onto the low-rank basis, such as $\langle \mathcal{L}(f), X_i \rangle_x$, would generally require constructing the full tensor, incurring a computational cost of $\mathcal{O}(N_x^{d_x} N_v^{d_v})$, thereby defeating the purpose of the low-rank approach. Resolving this bottleneck is the primary contribution of this work: in Section \ref{sec:discretization}, we propose a novel discrete formulation of $\mathcal{L}$ that recovers a separable structure, enabling efficient DLR evolution.

\section{Separable discretization of the Fokker--Planck operator and its low-rank projection}
\label{sec:discretization}

In this section, we construct a conservative finite-difference approximation of the Fokker--Planck operator \eqref{eq:op_L_def} that 
decouples the spatial and velocity components. We then project this discrete operator onto the low-rank basis, demonstrating how this separable structure enables efficient DLR evolution.

We consider one spatial and one velocity dimension (1D1V) for simplicity; the approach can be extended to higher dimensions straightforwardly.\footnote{\modifysecond{In the 1D1V setting used throughout the rest of this paper, the discrete distribution $f_{p,q}$ is a matrix rather than a tensor as referred to in Section~\ref{sec:dlr_algorithm}. However, we retain the term ``tensor" for consistency with the higher-dimensional formulation.}} We assume that the spatial and velocity domains are $[x_{\min},x_{\max}]$ and $[v_{\min},v_{\max}]$, respectively, with grid points 
given by
\[
  x_p \;=\; x_{\min} + \Big(p-\tfrac{1}{2}\Big)\,\Delta x, \quad p=1,\dots,N_x,
  \qquad
  v_q \;=\; v_{\min} + \Big(q-\tfrac{1}{2}\Big)\,\Delta v, \quad q=1,\dots,N_v,
\]
where $\Delta x=(x_{\max}-x_{\min})/N_x$ and $\Delta v=(v_{\max}-v_{\min})/N_v$.
We impose periodic boundary condition in \(x\) and zero-flux boundary condition in \(v\). The discrete low-rank ansatz for the distribution \(f\) at grid point \((x_p,v_q)\) is
\begin{equation}\label{eq:lr_ansatz_grid}
  f_{p,q}(t) \;=\; \sum_{i=1}^r \sum_{j=1}^r X_{i,p}(t)\,S_{ij}(t)\,V_{j,q}(t),
\end{equation}
where the basis vectors are orthonormal with respect to the discrete inner products defined by
\[
  \langle \phi, \psi \rangle_x^h = \sum_{p=1}^{N_x} \phi_p\,\psi_p\,\Delta x,
  \qquad
  \langle \phi, \psi \rangle_v^h = \sum_{q=1}^{N_v} \phi_q\,\psi_q\,\Delta v.
\]
The full phase-space inner product is given by \[\langle \phi, \psi \rangle_{xv}^h = \sum_{p,q} \phi_{p,q}\,\psi_{p,q}\,\Delta x\,\Delta v.\]

\subsection{Construction of the separable stencil}
For the 1D1V Fokker--Planck operator $\mathcal{L}(f) = \partial_v (M \partial_v (f/M))$, we define its discretization at spatial grid point \(x_p\) and velocity grid point \(v_q\) as:
\begin{equation}\label{eq:LFP_semi_disc}
  [\mathcal{L}_{h}(f)]_{p,q}
  \;=\;\frac{1}{\Delta v}\Big(\mathcal{J}_{p,q+1/2}-\mathcal{J}_{p,q-1/2}\Big),
\end{equation}
where the flux $\mathcal{J}_{p,q+1/2}$ is given by
\begin{equation}\label{eq:flux_half_def}
  \mathcal{J}_{p,q+1/2}
  \;=\; M_{p,q+1/2}\,\frac{(f/M)_{p,q+1}-(f/M)_{p,q}}{\Delta v},
  \qquad
  M_{p,q+1/2}\;=\;\sqrt{M_{p,q}\,M_{p,q+1}}.
\end{equation}
Note that we use a geometric-mean interpolation for the Maxwellian at the half grid point \cite{JY11}. Other interpolations, such as the arithmetic mean $M_{p,q+1/2} = \frac{1}{2}(M_{p,q} + M_{p,q+1})$, are also possible, but would result in more terms in the separable expansion.

Substituting \eqref{eq:flux_half_def} into \eqref{eq:LFP_semi_disc} yields the three-point conservative stencil:
\begin{equation}\label{eq:fp_stencil_form}
  [\mathcal{L}_{h}(f)]_{p,q}
  \;=\;\frac{1}{\Delta v^2} \left[
    \sqrt{\frac{M_{p,q}}{M_{p,q+1}}}\,f_{p,q+1}
    -\Big(\sqrt{\frac{M_{p,q+1}}{M_{p,q}}}+\sqrt{\frac{M_{p,q-1}}{M_{p,q}}}\Big)\,f_{p,q}
    +\sqrt{\frac{M_{p,q}}{M_{p,q-1}}}\,f_{p,q-1}
  \right].
\end{equation}
Using $v_{q+1}-v_q=\Delta v$ and $M_{p,q} \propto \exp (-(v_q-E_p)^2/2)$, we obtain
\begin{align}
  \sqrt{\frac{M_{p,q}}{M_{p,q+1}}}
  &= \exp \left(\frac{(v_{q+1}-E_p)^2-(v_q-E_p)^2}{4}\right)
   \;=\; \underbrace{\exp \Big(\tfrac{\Delta v}{4}\,(v_q+v_{q+1})\Big)}_{\displaystyle :=\alpha_{q+1/2}}
       \;\underbrace{\exp \Big(-\tfrac{\Delta v}{2}\,E_p\Big)}_{\displaystyle :=\beta_p}.\label{eq:ratio_plus}
\end{align}
Hence, it factors cleanly into space-only and velocity-only terms. Using the notation $\alpha_{q+1/2}$ and $\beta_p$, \eqref{eq:fp_stencil_form} can be written as
\begin{equation}\label{eq:fp_discrete_separable}
  [\mathcal{L}_{h}(f)]_{p,q}
  \;=\;\frac{1}{\Delta v^2} \left[
     \alpha_{q+1/2}\beta_p\,f_{p,q+1}
     -\Big(\frac{1}{\alpha_{q+1/2}\beta_p}+\alpha_{q-1/2}\beta_p\Big)\,f_{p,q}
     +\frac{1}{\alpha_{q-1/2} \beta_p} \,f_{p,q-1}
  \right],
\end{equation}
which is a fully separable stencil.


\subsection{Low-rank projection}
\label{sec:fp_projection}

With the discretization \eqref{eq:fp_discrete_separable}, we now show that the projections onto the low-rank basis, as described in the previous section, can all be evaluated efficiently. For brevity, we focus only on the terms that involve the Fokker--Planck operator.

\paragraph{K-step.}
We compute the projection of the discrete operator $\mathcal{L}_h(f)$ onto the velocity basis $V_j$. Substituting the ansatz $f_{p,q} = \sum_{\ell=1}^r K_{\ell,p} V_{\ell,q}$ into \eqref{eq:fp_discrete_separable}, the $p$-th component of the projected term reads:
\begin{equation}
\begin{aligned}
      \Big[\langle \mathcal{L}_{h}(f), V_j \rangle_v^h\Big]_p 
      &= \sum_{\ell=1}^r K_{\ell,p} \Bigg[ 
      \frac{\beta_p}{\Delta v^2} \sum_{q=1}^{N_v} \left( \alpha_{q+1/2} V_{\ell,q+1} - \alpha_{q-1/2} V_{\ell,q} \right)V_{j,q} \Delta v \\
      &\qquad \qquad + \frac{1}{\beta_p \Delta v^2} \sum_{q=1}^{N_v} \left( \frac{V_{\ell,q-1}}{\alpha_{q-1/2}} - \frac{V_{\ell,q}}{\alpha_{q+1/2}} \right)V_{j,q} \Delta v
      \Bigg].
\end{aligned}
\end{equation}

\paragraph{S-step.}
We compute the projection of the discrete operator $\mathcal{L}_h(f)$ onto the tensor product basis $X_i V_j$. Substituting $f_{p,q} = \sum_{k,\ell} X_{k,p} S_{k\ell} V_{\ell,q}$ into \eqref{eq:fp_discrete_separable}, the double summation over space and velocity decouples:
\begin{equation}
\begin{aligned}
    \big\langle \mathcal{L}_{h}(f), X_i V_j \big\rangle_{xv}^h
    &= \sum_{k,\ell=1}^r S_{k\ell} \Bigg[
    \left(\sum_{p=1}^{N_x} \beta_p X_{k,p} X_{i,p} \Delta x\right) \left( \frac{1}{\Delta v^2}\sum_{q=1}^{N_v} \left( \alpha_{q+1/2} V_{\ell,q+1} - \alpha_{q-1/2} V_{\ell,q} \right)V_{j,q}\Delta v \right) \\
    &\quad + \left(\sum_{p=1}^{N_x} \frac{X_{k,p} X_{i,p}}{\beta_p} \Delta x\right) \left( \frac{1}{\Delta v^2}\sum_{q=1}^{N_v} \left( \frac{V_{\ell,q-1}}{\alpha_{q-1/2}} - \frac{V_{\ell,q}}{\alpha_{q+1/2}} \right)V_{j,q}\Delta v \right)
    \Bigg].
\end{aligned}
\end{equation}

\paragraph{L-step.}
We compute the projection of the discrete operator $\mathcal{L}_h(f)$  onto the spatial basis $X_i$. Substituting $f_{p,q} = \sum_{k=1}^r X_{k,p} L_{k,q}$ into \eqref{eq:fp_discrete_separable}, the $q$-th component of the projected term reads:
\begin{equation}
\begin{aligned}
    \Big[\langle \mathcal{L}_{h}(f), X_i \rangle_x^h\Big]_q
    &= \frac{1}{\Delta v^2} \sum_{k=1}^r \Bigg[
    \left(\sum_{p=1}^{N_x} \beta_p X_{k,p} X_{i,p} \Delta x\right) \left( \alpha_{q+1/2} L_{k,q+1} - \alpha_{q-1/2} L_{k,q} \right) \\
    &\quad + \left(\sum_{p=1}^{N_x} \frac{X_{k,p} X_{i,p}}{\beta_p} \Delta x\right) \left( \frac{L_{k,q-1}}{\alpha_{q-1/2}} - \frac{L_{k,q}}{\alpha_{q+1/2}} \right)
    \Bigg].
\end{aligned}
\end{equation}

\section{Time discretization and the fully discrete low-rank scheme}
\label{sec:time_discretization}

Equation \eqref{eq:VPFP-1} involves an advection operator and a diffusive-type Fokker--Planck operator, the latter becoming stiff in the limit $\varepsilon \to 0$.
A robust time integrator must handle this stiffness without imposing overly restrictive time steps, while preserving the correct asymptotic behavior.
In the DLR setting, standard AP schemes cannot be applied straightforwardly;
the projector-splitting structure, specifically the backward S-step, requires special stability considerations \cite{zhang2025stability}.
In this section, we propose first-order and second-order low-rank IMEX schemes.
We employ the discretize-then-project strategy discussed in the previous section, but omit the subscripts $p,q$ for simplicity.

\subsection{A first-order scheme}
\label{sec:first_order_schemes}


We first present a simple first-order IMEX scheme in the full-tensor case, which serves as a basis for constructing the low-rank scheme. Let $\mathcal{A}_h$ denote the discrete advection operator and $\mathcal{L}_h[E]$ denote the discrete Fokker--Planck operator constructed using the electric field $E$. A simple first-order IMEX scheme reads
\begin{equation}\label{eq:full_tensor_imex}
  \frac{f^{n+1} - f^{n}}{\Delta t} = \mathcal{A}_h(f^{n}) + \frac{1}{\varepsilon}\mathcal{L}_h[E^{n+1}](f^{n+1}).
\end{equation}
Due to the dependence on $E^{n+1}$, directly solving for $f^{n+1}$ would require iterations. To get around this, one can use the local mass conservation property, $\langle \mathcal{L}_h[E^{n+1}](f^{n+1}), 1 \rangle_v^h = \mathbf{0}$. Taking the zeroth moment of \eqref{eq:full_tensor_imex} yields 
\begin{equation}\label{eq:rho_explicit_update}
  \frac{\rho^{n+1} - \rho^n}{\Delta t} = \langle \mathcal{A}_h(f^n), 1 \rangle_v^h, \quad \rho^n := \langle f^n, 1 \rangle_v^h.
\end{equation}
Therefore, one can first compute $\rho^{n+1}$ via \eqref{eq:rho_explicit_update}, then solve the Poisson equation \eqref{eq:Poisson} to obtain $E^{n+1}$. With $E^{n+1}$ known, the scheme \eqref{eq:full_tensor_imex} can be easily solved by inverting a linear operator.

\subsubsection{A first-order low-rank IMEX scheme}
\label{sec:first_order_DLR}

In the low-rank framework, we first use \eqref{eq:rho_explicit_update} to obtain a prediction of the density, and then use this predicted density in the subsequent K-S-L steps to update the distribution function. Suppose that we start from the low-rank decomposition $f^n=\sum_{i,j=1}^r X_i^nS_{ij}^nV_j^n$. The predicted density is then computed as
\begin{equation}\label{eq:rho_hat_update}
    \hat{\rho}^{n+1} = \rho^n + \Delta t \langle \mathcal{A}_h(f^n), 1 \rangle_v^h.
\end{equation}
We distinguish the predicted density $\hat{\rho}^n$ from the true density $\rho^n$, defined through the moment of $f^n$, because the low-rank scheme introduced below does not strictly preserve mass. The Poisson equation is then solved using $\hat{\rho}^{n+1}$ to obtain $E^{n+1}$.

With the field $E^{n+1}$ known, the Fokker--Planck operator becomes linear $\mathcal{L}_h^{n+1} := \mathcal{L}_h[E^{n+1}]$. We then apply the projector-splitting integrator to \eqref{eq:VPFP-1}. However, applying the IMEX scheme  \eqref{eq:full_tensor_imex} across all K-S-L steps leads to instability, because the S-step evolves backward in time. \modifysecond{The backward-in-time character of the S-step has been observed in earlier studies of DLR integrators in various settings, motivating the BUG integrators that avoid the backward S-step entirely~\cite{ceruti2022unconventional,CKL22}.  Instability has also been reported empirically for projector-splitting IMEX schemes applied to a kinetic equation with stiff collisions in~\cite{coughlin2022efficient}.  Subsequent stability analyses include \cite{kusch2023stability} for for hyperbolic problems and, most directly relevant to the present setting, an $L^2$-stability analysis in~\cite{zhang2025stability} of a hybrid backward/forward-Euler projector-splitting scheme for a parabolic prototype equation that mimics the diffusion structure of the Fokker--Planck operator.}  Motivated by \cite{zhang2025stability}, we propose a hybrid IMEX scheme: the K and L steps use an implicit (backward Euler) update for the Fokker--Planck operator, while the S step uses an explicit (forward Euler) update. The algorithm is summarized in Algorithm~\ref{alg:first_order}.

\medskip
\noindent\textbf{1. K-step.} Define $K_j^n=\sum_{i=1}^r X_i^nS_{ij}^n$, and solve for $K_j^{(1)}$ using
\begin{equation}\label{eq:K_fd_imex}
  \frac{K_{j}^{(1)} - K_{j}^{n}}{\Delta t} = \langle \mathcal{A}_h(f^{n}), V_{j}^{n}\rangle_v^h + \frac{1}{\varepsilon}\langle \mathcal{L}_h^{n+1}(f^{(1)}), V_{j}^{n}\rangle_v^h,
\end{equation}
where $f^{(1)} = \sum_{j=1}^r K_{j}^{(1)} V_{j}^n$. Computing $\{K^{(1)}_j\}_{j=1}^{r}$ requires solving $N_x$ decoupled $r\times r$ linear systems,
one for each spatial grid point. A QR factorization of $\{K^{(1)}_j\}_{j=1}^{r}$ yields the new spatial basis $\{X^{n+1}_i\}_{i=1}^{r}$ and intermediate coefficient matrix $S^{(1)}$.

\noindent\textbf{2. S-step.} Solve for $S^{(2)}$ using
\begin{equation}\label{eq:S_fd_forward}
  \frac{S_{ij}^{(2)} - S_{ij}^{(1)}}{\Delta t} = - \langle \mathcal{A}_h(f^{(1)}), X_{i}^{n+1} V_{j}^{n} \rangle_{xv}^h - \frac{1}{\varepsilon}\langle \mathcal{L}_h^{n+1}(f^{(1)}), X_{i}^{n+1} V_{j}^{n} \rangle_{xv}^h,
\end{equation}
where $f^{(1)} = \sum_{i,j} X_{i}^{n+1} S_{ij}^{(1)} V_{j}^{n}$.
Note that this is a fully explicit step.

\noindent\textbf{3. L-step.} Define $L_{i}^{(2)} = \sum_{j=1}^r S_{ij}^{(2)} V_j^n$, and solve for $L_i^{n+1}$ using
\begin{equation}\label{eq:L_fd_imex}
  \frac{L_{i}^{n+1} - L_{i}^{(2)}}{\Delta t} = \langle \mathcal{A}_h(f^{(2)}), X_{i}^{n+1}\rangle_x^h + \frac{1}{\varepsilon}\langle \mathcal{L}_h^{n+1}(f^{n+1}), X_{i}^{n+1}\rangle_x^h,
\end{equation}
where $f^{(2)} = \sum_{i=1}^r X_{i}^{n+1} L_{i}^{(2)}$ and $f^{n+1}=\sum_{i=1}^r X_i^{n+1}L_i^{n+1}$. Computing $\{L_i^{n+1}\}_{i=1}^r$ requires solving a linear system of size $(rN_v)\times(rN_v)$, which is block-tridiagonal in the velocity grid with $r\times r$ blocks. A final QR factorization of $\{L^{n+1}_i\}_{i=1}^{r}$ yields the new velocity basis $\{V^{n+1}_j\}_{j=1}^{r}$ and coefficient matrix $S^{n+1}$.

\begin{algorithm}
\caption{First-order low-rank IMEX scheme}
\label{alg:first_order}
\begin{algorithmic}[1]
\Require Rank-$r$ factors $\{X^n_i\}_{i=1}^{r}, S^n, \{V^n_j\}_{j=1}^{r}$ at $t^n$.
\Ensure Rank-$r$ factors $\{X^{n+1}_i\}_{i=1}^{r}, S^{n+1}, \{V^{n+1}_j\}_{j=1}^{r}$ at $t^{n+1}$.

\AlgPhase{Step 1: Density prediction}
    \State Compute density $\rho^n = \langle f^n, 1 \rangle_v^h$ and the predicted density $\hat{\rho}^{n+1} = \rho^n + \Delta t \langle \mathcal{A}_h(f^n), 1 \rangle_v^h$.
    \State Solve the Poisson equation using $\hat{\rho}^{n+1}$ for $E^{n+1}$.
    \State Define the Fokker--Planck operator $\mathcal{L}_h^{n+1} := \mathcal{L}_h[E^{n+1}]$.

\AlgPhase{Step 2: Lie--Trotter projector-splitting evolution}
    \State \textbf{K-step}: Update $\{K_j\}_{j=1}^{r}$ via Eq.~\eqref{eq:K_fd_imex}.
    \State \phantom{\textbf{K-step}} QR factorization yields $\{X^{n+1}_i\}_{i=1}^{r}$ and $S^{(1)}$.
    \State \textbf{S-step}: Update $S$ from $S^{(1)}$ to $S^{(2)}$ via Eq.~\eqref{eq:S_fd_forward}.
    \State \textbf{L-step}: Update $\{L_i\}_{i=1}^{r}$ via Eq.~\eqref{eq:L_fd_imex}.
    \State \phantom{\textbf{L-step}} QR factorization yields $\{V^{n+1}_j\}_{j=1}^{r}$ and $S^{n+1}$.
\end{algorithmic}
\end{algorithm}

\subsection{A second-order scheme}
\label{sec:second_order_schemes}

To achieve higher temporal accuracy, we employ a particular second-order IMEX scheme used in \cite{filbet2010class}. This scheme is special in that its implicit tableau corresponds to the trapezoidal rule, which is required in the stability analysis established in
\cite{zhang2025stability} for applying DLR methods to stiff parabolic problems. 

In the full-tensor case, the scheme reads
\begin{align}
    \frac{f^{n+1/2} - f^n}{\Delta t/2} &= \mathcal{A}_h(f^n) + \frac{1}{\varepsilon}\mathcal{L}_h[E^{n+1/2}](f^{n+1/2}),\label{eq:stage_coupled} \\
    \frac{f^{n+1} - f^n}{\Delta t} &= \mathcal{A}_h(f^{n+1/2}) + \frac{1}{2\varepsilon}\Big( \mathcal{L}_h[E^n](f^n) + \mathcal{L}_h[E^{n+1}](f^{n+1}) \Big).\label{eq:step_coupled}
\end{align}
To resolve the coupling with the unknown fields $E^{n+1/2}$ and $E^{n+1}$, similar strategy in the first-order scheme applies. 
Taking the zeroth moment of \eqref{eq:stage_coupled} yields
\begin{equation}\label{eq:dens_pred}
  \rho^{n+1/2} = \rho^{n} + \frac{\Delta t}{2} \langle \mathcal{A}_h(f^n), 1 \rangle_v^h.
\end{equation}
We solve the Poisson equation using $\rho^{n+1/2}$ to obtain $E^{n+1/2}$, allowing us to solve \eqref{eq:stage_coupled} efficiently for $f^{n+1/2}$.
Similarly, taking the zeroth moment of \eqref{eq:step_coupled} yields 
\begin{equation}\label{eq:dens_corr}
  \rho^{n+1} = \rho^{n} + \Delta t \langle \mathcal{A}_h(f^{n+1/2}), 1 \rangle_v^h.
\end{equation}
With $\rho^{n+1}$ (and thus $E^{n+1}$) determined, we finally solve \eqref{eq:step_coupled} to obtain $f^{n+1}$.

\subsubsection{A second-order low-rank IMEX scheme}
\label{sec:lowrank_second_order}

In the low-rank framework, we first compute the predicted density at $t_{n+1/2}$:
\begin{equation}
    \hat{\rho}^{n+1/2} = \rho^n + \frac{\Delta t}{2} \langle \mathcal{A}_h(f^n), 1 \rangle_v^h.
\end{equation}
We then solve the Poisson equation using $\hat{\rho}^{n+1/2}$ to obtain $E^{n+1/2}$. To obtain $E^{n+1}$, we need $f^{n+1/2}$ which can be computed using the first-order low-rank scheme (Algorithm \ref{alg:first_order}) with step size $\Delta t/2$. Then 
\begin{equation}
    \hat{\rho}^{n+1} = \rho^n + \Delta t \langle \mathcal{A}_h(f^{n+1/2}), 1 \rangle_v^h.
\end{equation}
Solving the Poisson equation using $\hat{\rho}^{n+1}$ yields $E^{n+1}$.
With the field $E^{n+1/2}$ and $E^{n+1}$ known, the Fokker--Planck operator becomes $\mathcal{L}_h^{n+1/2}:=\mathcal{L}_h[E^{n+1/2}]$ and $\mathcal{L}_h^{n+1}:=\mathcal{L}_h[E^{n+1}]$.

We then apply the projector-splitting integrator to \eqref{eq:VPFP-1} using Strang splitting
\[\mathcal{K}(\tfrac{\Delta t}{2}) \circ \mathcal{S}(\tfrac{\Delta t}{2}) \circ \mathcal{L}(\Delta t) \circ \mathcal{S}(\tfrac{\Delta t}{2}) \circ \mathcal{K}(\tfrac{\Delta t}{2}).\]
In each substep, we employ the second-order IMEX scheme discussed above. The algorithm is summarized in Algorithm~\ref{alg:second_order}.


\medskip
\noindent\textbf{1. First K-step ($\Delta t/2$).}
Define $K^n_j = \sum_{i=1}^r X^n_i S^n_{ij}$, and solve for $K^{(1)}$ using
\begin{align}
    \frac{K^{n,\ast}_{j} - K^{n}_{j}}{\Delta t/4} &= \langle \mathcal{A}_h(f[K^{n}]),V^n_j\rangle_v^h + \frac{1}{\varepsilon}\langle \mathcal{L}_{h}^{n+1/2}(f[K^{n,\ast}]),V^n_j\rangle_v^h, \\
    \frac{K^{(1)}_{j} - K^{n}_{j}}{\Delta t/2} &= \langle \mathcal{A}_h(f[K^{n,\ast}]),V^n_j\rangle_v^h + \frac{1}{2\varepsilon}\langle \mathcal{L}_{h}^{n}(f[K^{n}]) + \mathcal{L}_{h}^{n+1}(f[K^{(1)}]),V^n_j\rangle_v^h.
\end{align}
Here, $f[K] = \sum_{\ell} K_\ell V^n_\ell$. A QR factorization of $\{K^{(1)}_j\}_{j=1}^{r}$ yields the intermediate spatial basis $\{X^{(1)}_i\}_{i=1}^{r}$ and coefficient matrix $S^{(1)}$.

\noindent\textbf{2. First S-step ($\Delta t/2$).}
Solve for $S^{(2)}$ using
\begin{align}
    \frac{S^{(1),\ast}_{ij} - S^{(1)}_{ij}}{\Delta t/4} &= - \big\langle \mathcal{A}_h(f[S^{(1)}]),X^{(1)}_i V^n_j \big\rangle_{xv}^h - \frac{1}{\varepsilon}\big\langle \mathcal{L}_{h}^{n+1/2}(f[S^{(1),\ast}]),X^{(1)}_i V^n_j \big\rangle_{xv}^h, \\
    \frac{S^{(2)}_{ij} - S^{(1)}_{ij}}{\Delta t/2} &= - \big\langle \mathcal{A}_h(f[S^{(1),\ast}]),X^{(1)}_i V^n_j \big\rangle_{xv}^h - \frac{1}{2\varepsilon}\big\langle \mathcal{L}_{h}^{n}(f[S^{(1)}]) + \mathcal{L}_{h}^{n+1}(f[S^{(2)}]),X^{(1)}_i V^n_j \big\rangle_{xv}^h.
\end{align}
Here, $f[S] = \sum_{k,\ell} X^{(1)}_k S_{k\ell} V^n_\ell$.

\noindent\textbf{3. L-step ($\Delta t$).}
Define $L^{(2)}_i = \sum_{j=1}^r S^{(2)}_{ij} V^n_j$, and solve for $L^{(3)}$ using
\begin{align}
    \frac{L^{(2),\ast}_{i} - L^{(2)}_{i}}{\Delta t/2} &= \langle \mathcal{A}_h(f[L^{(2)}]),X^{(1)}_i\rangle_x^h + \frac{1}{\varepsilon}\langle \mathcal{L}_{h}^{n+1/2}(f[L^{(2),\ast}]),X^{(1)}_i\rangle_x^h, \\
    \frac{L^{(3)}_{i} - L^{(2)}_{i}}{\Delta t} &= \langle \mathcal{A}_h(f[L^{(2),\ast}]),X^{(1)}_i\rangle_x^h + \frac{1}{2\varepsilon}\langle \mathcal{L}_{h}^{n}(f[L^{(2)}]) + \mathcal{L}_{h}^{n+1}(f[L^{(3)}]),X^{(1)}_i\rangle_x^h.
\end{align}
Here, $f[L] = \sum_{k} X^{(1)}_k L_k$. A QR factorization of $\{L^{(3)}_i\}_{i=1}^{r}$ yields the new velocity basis $\{V^{n+1}_j\}_{j=1}^{r}$ and coefficient matrix $S^{(3)}$.

\noindent\textbf{4. Second S-step ($\Delta t/2$).}
Solve for $S^{(4)}$ using
\begin{align}
    \frac{S^{(3),\ast}_{ij} - S^{(3)}_{ij}}{\Delta t/4} &= - \big\langle \mathcal{A}_h(f[S^{(3)}]),X^{(1)}_i V^{n+1}_j \big\rangle_{xv}^h - \frac{1}{\varepsilon}\big\langle \mathcal{L}_{h}^{n+1/2}(f[S^{(3),\ast}]),X^{(1)}_i V^{n+1}_j \big\rangle_{xv}^h, \\
    \frac{S^{(4)}_{ij} - S^{(3)}_{ij}}{\Delta t/2} &= - \big\langle \mathcal{A}_h(f[S^{(3),\ast}]),X^{(1)}_i V^{n+1}_j \big\rangle_{xv}^h - \frac{1}{2\varepsilon}\big\langle \mathcal{L}_{h}^{n}(f[S^{(3)}]) + \mathcal{L}_{h}^{n+1}(f[S^{(4)}]),X^{(1)}_i V^{n+1}_j \big\rangle_{xv}^h.
\end{align}
Here, $f[S] = \sum_{k,\ell} X^{(1)}_k S_{k\ell} V^{n+1}_\ell$.

\noindent\textbf{5. Second K-step ($\Delta t/2$).}
Define $K^{(4)}_j = \sum_{i=1}^r X^{(1)}_i S^{(4)}_{ij}$, and solve for $K^{n+1}$ using
\begin{align}
    \frac{K^{(4),\ast}_{j} - K^{(4)}_{j}}{\Delta t/4} &= \langle \mathcal{A}_h(f[K^{(4)}]),V^{n+1}_j\rangle_v^h + \frac{1}{\varepsilon}\langle \mathcal{L}_{h}^{n+1/2}(f[K^{(4),\ast}]),V^{n+1}_j\rangle_v^h, \\
    \frac{K^{n+1}_{j} - K^{(4)}_{j}}{\Delta t/2} &= \langle \mathcal{A}_h(f[K^{(4),\ast}]),V^{n+1}_j\rangle_v^h + \frac{1}{2\varepsilon}\langle \mathcal{L}_{h}^{n}(f[K^{(4)}]) + \mathcal{L}_{h}^{n+1}(f[K^{n+1}]),V^{n+1}_j\rangle_v^h.
\end{align}
Here, $f[K] = \sum_{\ell} K_\ell V^{n+1}_\ell$. A final QR factorization of $\{K^{n+1}_j\}_{j=1}^{r}$ yields the new spatial basis $\{X^{n+1}_i\}_{i=1}^{r}$ and coefficient matrix $S^{n+1}$.

\begin{algorithm}
\caption{Second-order low-rank IMEX scheme}
\label{alg:second_order}
\begin{algorithmic}[1]
\Require Rank-$r$ factors $\{X^n_i\}_{i=1}^{r}, S^n, \{V^n_j\}_{j=1}^{r}$ at $t^n$.
\Ensure Rank-$r$ factors $\{X^{n+1}_i\}_{i=1}^{r}, S^{n+1}, \{V^{n+1}_j\}_{j=1}^{r}$ at $t^{n+1}$.

\AlgPhase{Step 1: Density prediction}
    \State Compute density $\rho^n = \langle f^n, 1 \rangle_v^h$ and $\hat{\rho}^{n+1/2} = \rho^n + \frac{\Delta t}{2} \langle \mathcal{A}_h(f^n), 1 \rangle_v^h$.
    \State Solve the Poisson equation using $\hat{\rho}^{n+1/2}$ for $E^{n+1/2}$.
    \State Evolve $f^n \to f^{n+1/2}$ using Algorithm \ref{alg:first_order} for $\Delta t/2$ step.
    \State Compute $\hat{\rho}^{n+1} = \rho^n + \Delta t \langle \mathcal{A}_h(f^{n+1/2}), 1 \rangle_v^h$. 
    \State Solve the Poisson equation using $\hat{\rho}^{n+1}$ for $E^{n+1}$.
    \State Define $\mathcal{L}^n := \mathcal{L}[E^n]$, $\mathcal{L}^{n+1/2} := \mathcal{L}[E^{n+1/2}]$, $\mathcal{L}^{n+1} := \mathcal{L}[E^{n+1}]$.

\AlgPhase{Step 2: Strang projector-splitting evolution}
    \State \textbf{K-step} ($\Delta t/2$): Evolve $\{K_j\}_{j=1}^{r}$ via 2nd-order IMEX.
    \State \phantom{\textbf{1. K-step}} QR factorization yields $\{X^{(1)}_i\}_{i=1}^{r}$ and $S^{(1)}$.

    \State \textbf{S-step} ($\Delta t/2$): Evolve $S$ to obtain $S^{(2)}$ via 2nd-order IMEX.

    \State \textbf{L-step} ($\Delta t$): \phantom{/2} Evolve $\{L_i\}_{i=1}^{r}$ via 2nd-order IMEX.
    \State \phantom{\textbf{3. L-step}} QR factorization yields $\{V^{n+1}_j\}_{j=1}^{r}$ and $S^{(3)}$.

    \State \textbf{S-step} ($\Delta t/2$): Evolve $S$ to obtain $S^{(4)}$ via 2nd-order IMEX.

    \State \textbf{K-step} ($\Delta t/2$): Evolve $\{K_j\}_{j=1}^{r}$ via 2nd-order IMEX.
    \State \phantom{\textbf{5. K-step}} QR factorization yields $\{X^{n+1}_i\}_{i=1}^{r}$ and $S^{n+1}$.
\end{algorithmic}
\end{algorithm}

\section{Asymptotic--preserving (AP) analysis}
\label{sec:AP}

In this section, we analyze the asymptotic behavior of the proposed first-order low-rank IMEX scheme as $\varepsilon \to 0$. Our goal is to establish that the solution relaxes to the discrete Maxwellian determined by the macroscopic field, with a residual error controlled by $\varepsilon$ and spatial field fluctuations.

We focus our analysis on the first-order scheme because it employs backward Euler time integration for the Fokker--Planck operator during the K and L steps. In the stiff limit, backward Euler is $L$-stable and acts as a projection onto the local equilibrium manifold, yielding a strongly AP scheme. This allows us to prove a global relaxation property via coercivity estimates. 


\subsection{Notation and micro-macro decomposition}

We work in the 1D1V discrete setting and utilize the discrete norms and inner products defined in Section~\ref{sec:discretization}. Unless otherwise specified, operations involving the velocity inner product $\langle \cdot, \cdot \rangle_v^h$ are understood to map phase-space grid functions to spatial grid functions. Let $L_v = v_{\max} - v_{\min}$ denote the length of the velocity domain.

We assume without loss of generality that the discrete Maxwellian $M$ is normalized on the velocity grid:
\begin{equation}
    \langle M, 1 \rangle_v^h = \mathbf{1}.
\end{equation}
Under this normalization, the local equilibrium is given by $\rho M$, where the product implies row-wise scaling by the density vector.

To quantify the distance to equilibrium, we employ the standard micro-macro decomposition of the distribution function:
\begin{equation} \label{eq:micro_macro_def}
    f = \rho M + g.
\end{equation}
Here, $\rho M$ represents the projection of $f$ onto the kernel of the Fokker--Planck operator (the local equilibrium manifold), and $g$ represents the kinetic fluctuation. Due to the normalization of $M$ and definition of the macroscopic density $\rho = \langle f, 1 \rangle_v^h$, the fluctuation satisfies \modifyfirst{the zero velocity moment condition}:
\begin{equation}
    \langle g, 1 \rangle_v^h = \mathbf{0}.
\end{equation}
Our goal is to bound the norm of the fluctuation:
\begin{equation}
    \| f - \rho M \|_{xv}^h = \| g \|_{xv}^h.
\end{equation}

For the analysis, we fix the field $E^{n+1}$ and recall the space-dependent and velocity-dependent weights defined in \eqref{eq:ratio_plus}:
\begin{equation}\label{eq:AP_weights}
  \beta_p = \exp \Big(-\tfrac{\Delta v}{2}\,E^{n+1}_p\Big), \qquad \alpha_{q+1/2} = \exp \Big(\tfrac{\Delta v}{4}\,(v_q+v_{q+1})\Big).
\end{equation}
We assume we are in a regime where the spatial variation of the field is controlled. We denote the fluctuation of the spatial weights by $\delta \beta_p := \beta_p - \bar{\beta}$ relative to the spatial arithmetic mean $\bar{\beta}$, with $\|\delta \beta\|_\infty = \max_p |\beta_p - \bar{\beta}|$. This fluctuation serves as a proxy for the local rank of the Maxwellian; $\delta \beta \equiv 0$ corresponds to a spatially constant field where the Maxwellian is exactly separable (rank-1).

\subsection{Key operator properties}
\label{sec:aux_lemmas}
We establish the key spectral properties of the discrete operator required for the main theorem. To obtain a coercive estimate for the Fokker--Planck operator, it is necessary to work with the relative fluctuation $g/M$.

\begin{proposition}[Discrete coercivity]
\label{prop:coercivity}
Consider the decomposition $f = \rho M + g$ \modifyfirst{with $\rho=\langle f,1\rangle_v^h$, so that the fluctuation $g$ satisfies the zero velocity moment condition $\langle g, 1 \rangle_v^h = \mathbf{0}$}. There exists a constant $\gamma_h > 0$ such that the Fokker--Planck operator satisfies the coercivity estimate:
\begin{equation}
  -\big\langle \mathcal L_{h}(f),\, g/M \big\rangle_{xv}^h \;\ge\; \gamma_h\,\big\| g/M \big\|_{xv}^{h,2},
\end{equation}
where the constant $\gamma_h$ is defined as:
\begin{equation}
  \gamma_h:=\frac{M_{\min}\,\lambda_N}{\Big(1+\sqrt{{L_v\,M_{\max}}}\Big)^2}, \qquad \text{with} \quad \lambda_N=\frac{4}{\Delta v^2}\sin^2\!\Big(\frac{\pi}{2N_v}\Big),
\end{equation}
and the Maxwellian bounds
\[
  M_{\min}:=\min_{p,q} M_{p,q+1/2},
  \qquad
  M_{\max}:=\max\!\Big\{ \max_{p,q} M_{p,q},\, \max_{p,q} M_{p,q+1/2}\Big\}.
\]
\end{proposition}

\begin{proof}
Let $u = g/M$ denote the relative fluctuation grid function. Since $\mathcal{L}_{h}(\rho M) = 0$, the action of the operator is determined solely by $g$. We test the operator against $u$. Using the summation by parts, we obtain the global dissipation relation:
\[
\big\langle \mathcal L_{h}(f),\,u \big\rangle_{xv}^h
= -\,\frac{\Delta x}{\Delta v}\sum_{p=1}^{N_x}\sum_{q=1}^{N_v-1}
M_{p,q+1/2}\,\big( u_{p,q+1}-u_{p,q} \big)^2.
\]
We now seek a lower bound for the term on the right-hand side. Let $\mathcal{D}(u)$ be the spatial grid function representing the local dissipation:
\[
  [\mathcal{D}(u)]_p := \frac{1}{\Delta v} \sum_{q=1}^{N_v-1} (u_{p,q+1}-u_{p,q})^2.
\]
Using the bound $M_{p,q+1/2} \ge M_{\min}$, we have the pointwise lower bound proportional to $\mathcal{D}(u)$. We apply the discrete Poincaré inequality on the uniform velocity grid. Let $\bar{u}$ denote the spatial vector of arithmetic means (i.e., $[\bar{u}]_p = \frac{1}{N_v}\sum_q u_{p,q}$). Then pointwise:
\[
  \mathcal{D}(u) \ge \lambda_N \| u - \bar{u} \|_{v}^{h,2}.
\]
To close the estimate, we need to relate the deviation from the arithmetic mean to the full norm. From the zero-moment condition $\langle g,1\rangle_v^h=\mathbf{0}$, we have
\begin{equation}\label{eq:zero_weighted_mean}
    \langle u, M \rangle_v^h = \langle g/M, M \rangle_v^h = \langle g, 1 \rangle_v^h = \mathbf{0}.
\end{equation}
Using the decomposition $u = (u - \bar{u}) + \bar{u}$, the normalization $\langle 1, M \rangle_v^h = \mathbf{1}$ and \eqref{eq:zero_weighted_mean}, we obtain:
\[
  \mathbf{0} = \langle u - \bar{u}, M \rangle_v^h + \bar{u} \langle 1, M \rangle_v^h \implies \bar{u} = - \langle u - \bar{u}, M \rangle_v^h.
\]
Applying the Cauchy--Schwarz inequality yields:
\[
  \bar{u}^2 = |\langle u - \bar{u}, M \rangle_v^h|^2 \le \|u - \bar{u}\|_{v}^{h,2} \, \|M\|_{v}^{h,2} \le M_{\max} \|u - \bar{u}\|_{v}^{h,2},
\]
where the last inequality uses the normalization $\langle 1, M \rangle_v^h = \mathbf{1}$ to bound $\|M\|_{v}^{h,2} = \sum_q M_q^2 \Delta v \le M_{\max}\sum_q M_q \Delta v = M_{\max}$.  Since $\bar u_p$ is constant in $v$, $\|\bar{u}\|_{v}^{h,2} = \bar{u}^2 L_v \le L_v M_{\max}\|u-\bar u\|_v^{h,2}$.
Substituting this back into the triangle inequality for $\|u\|_v^h$:
\[
  \|u\|_{v}^{h,2} \le \Big( \|u-\bar{u}\|_v^h + \|\bar{u}\|_v^h \Big)^2 \le \Big(1 + \sqrt{L_v M_{\max}}\Big)^2 \|u-\bar{u}\|_{v}^{h,2}.
\]
Combining the local estimates and summing over all spatial nodes $p$ yields:
\[
  -\big\langle \mathcal L_{h}(f),\,u \big\rangle_{xv}^h \ge M_{\min} \lambda_N \sum_{p=1}^{N_x} \|u - \bar{u}\|_{v}^{h,2} \Delta x \ge \frac{M_{\min} \lambda_N}{(1+\sqrt{L_v M_{\max}})^2} \|u\|_{xv}^{h,2}.
\]
This completes the proof.
\end{proof}

\begin{proposition}[Fluctuation--based projection bound]
\label{prop:fluctuation_bound}
Let $P_X$ be the spatial projector onto the subspace $\mathrm{span}\{X_i\}$. Let $\bar{\beta}$ denote the arithmetic mean of $\beta_p$ over the spatial domain, and define the fluctuation $\delta \beta_p := \beta_p - \bar{\beta}$ and the inverse relative amplitude $C_\beta := \max_p (\beta_p \bar{\beta})^{-1}$.
Then, for any distribution $f$ in the range of $P_X$, the orthogonal projection of the Fokker--Planck operator is bounded by:
\begin{equation}
  \|(I-P_X)\mathcal L_{h}(f)\|_{xv}^h \le \kappa_h \|f\|_{xv}^h,
\end{equation}
where the fluctuation constant is defined by:
\begin{equation}\label{eq:kappa_h_def}
  \kappa_h \;:=\; \|\delta \beta\|_\infty \left( \frac{2\alpha_{\max}}{\Delta v^2} \;+\; C_\beta \frac{2}{\alpha_{\min}\Delta v^2} \right).
\end{equation}
\end{proposition}

\begin{proof}
Recall that the separable stencil \eqref{eq:fp_discrete_separable} defines the operator $\mathcal{L}_{h}$. We observe that this operator acts as a sum of tensor products involving multiplication by spatial weights $\beta$ and $\beta^{-1}$:
\[
  \mathcal{L}_{h} \;=\; \beta \otimes \mathcal{T}^{(\alpha)} \;+\; \beta^{-1} \otimes \mathcal{T}^{(1/\alpha)}.
\]
The velocity difference operators $\mathcal{T}^{(\alpha)}$ and $\mathcal{T}^{(1/\alpha)}$ act on a velocity grid function $w$ as the two tensor factors arising directly from grouping~\eqref{eq:fp_discrete_separable} by powers of $\beta_p$:
\begin{align*}
    (\mathcal{T}^{(\alpha)} w)_q &= \frac{1}{\Delta v^2}\big(\alpha_{q+1/2}\,w_{q+1} - \alpha_{q-1/2}\,w_q\big), \\
    (\mathcal{T}^{(1/\alpha)} w)_q &= \frac{1}{\Delta v^2}\big(\alpha_{q-1/2}^{-1}\,w_{q-1} - \alpha_{q+1/2}^{-1}\,w_q\big).
\end{align*}
Since the projector $P_X$ acts only on the spatial coordinate, we can bound the two tensor terms separately.

Consider the first term involving $\beta$. We use the property that the scalar operator $\bar{\beta}I$ maps the subspace $\mathrm{ran}(P_X)$ into itself, and that $\mathcal{T}^{(\alpha)}$ acts only on velocity, hence preserves $\mathrm{ran}(P_X)$. Thus, for any $f \in \mathrm{ran}(P_X)$, $(I-P_X)(\bar{\beta}\,\mathcal{T}^{(\alpha)}f) = 0$. We write:
\[
  (I-P_X)\big(\beta\,\mathcal{T}^{(\alpha)}f\big) = (I-P_X)\big((\beta - \bar{\beta})\,\mathcal{T}^{(\alpha)}f\big).
\]
Since $\|I-P_X\|_{x}^{h} \le 1$, we bound the norm of the above:
\[
  \|(I-P_X)(\beta\,\mathcal{T}^{(\alpha)}f)\|_{xv}^h \;\le\; \|(\beta - \bar{\beta})\,\mathcal{T}^{(\alpha)}f\|_{xv}^h \;\le\; \|\delta \beta\|_\infty \,\|\mathcal{T}^{(\alpha)}\|_v^h\, \|f\|_{xv}^h.
\]

For the second term involving $\beta^{-1}$, we similarly choose the scalar constant $\bar{\beta}^{-1}$. The spatial coefficient variation is bounded algebraically:
\[
  |\beta_p^{-1} - \bar{\beta}^{-1}| \;=\; \left| \frac{\bar{\beta} - \beta_p}{\beta_p \bar{\beta}} \right| \;=\; \frac{1}{\beta_p \bar{\beta}} |\beta_p - \bar{\beta}| \;\le\; C_\beta \|\delta \beta\|_\infty.
\]
Thus, $\|(I-P_X)(\beta^{-1}\,\mathcal{T}^{(1/\alpha)}f)\|_{xv}^h \le C_\beta \|\delta \beta\|_\infty\,\|\mathcal{T}^{(1/\alpha)}\|_v^h\, \|f\|_{xv}^h$.

Finally, we bound the velocity operators. By applying the triangle inequality to the difference stencils defined above, we observe that each operator acts on neighboring grid points with weights bounded by $\alpha_{\max}$ (or $\alpha_{\min}^{-1}$). A direct estimation of the induced norms yields:
\[
  \|\mathcal{T}^{(\alpha)}\|_v^h \le \frac{2\alpha_{\max}}{\Delta v^2} \quad \text{and} \quad \|\mathcal{T}^{(1/\alpha)}\|_v^h \le \frac{2}{\alpha_{\min}\Delta v^2}.
\]
Combining these estimates with the spatial bounds yields the total constant $\kappa_h$ defined in \eqref{eq:kappa_h_def}.
\end{proof}

\subsection{Main results}

We now prove the AP property. We assume the formal asymptotic regime $0 < \varepsilon \ll \Delta t \ll 1$, wherein the numerical solution $f$ and its discrete derivatives are treated as bounded (smooth) $\mathcal{O}(1)$ quantities. \modifysecond{Throughout this analysis the spatial and velocity grids are fixed; consequently, constants such as $\gamma_h$ and $\kappa_h$ may depend on $\Delta x$ and $\Delta v$, but are required to be independent of $\varepsilon$.} Additionally, we require that the field fluctuation is small relative to the dissipation gap:

\begin{assumption}[Small--fluctuation regime]\label{asm:small_fluct}
We assume that the field fluctuation $\kappa_h$ (defined in Proposition~\ref{prop:fluctuation_bound}) satisfies
\begin{equation}\label{eq:regime_condition}
    \kappa_h \, M_{\max}^{n+1} \;<\; \gamma_h,
\end{equation}
where $\gamma_h$ is the coercivity constant from Proposition~\ref{prop:coercivity} and $M_{\max}^{n+1} := \max\!\big\{ \max_{p,q} M_{p,q}^{n+1},\, \max_{p,q} M_{p,q+1/2}^{n+1}\big\}$, an upper bound that controls both the cell-centered and half-grid Maxwellian values used in the proof.
\end{assumption}

\begin{theorem}[AP property up to spatial fluctuation]
\label{thm:main_AP}
Let $f^{n+1}$ be the low--rank solution obtained after the L--step update. Let $\rho^{n+1} = \langle f^{n+1}, 1 \rangle_v^h$ be the macroscopic density vector. Let $G^{n+1}$ denote the discrete residual defined by the time difference and advection operators (see \eqref{eq:G_derivation}). Under Assumption~\ref{asm:small_fluct} and the formal asymptotic regime assumptions (implying $\|G^{n+1}\|_{xv}^h = \mathcal{O}(1)$), let $\Theta := \gamma_h / (\gamma_h - \kappa_h M_{\max}^{n+1}) \ge 1$. Then, the Fokker--Planck residual satisfies
\begin{equation}\label{eq:residual_bound}
  \big\|\mathcal L_{h}^{n+1}(f^{n+1})\big\|_{xv}^h
  \;\le\;
  \Theta \Big( \varepsilon\,\|G^{n+1}\|_{xv}^h \;+\; \kappa_h\,M_{\max}^{n+1}\,\sqrt{L_v}\,\big\|\rho^{n+1}\big\|_x^h \Big),
\end{equation}
and 
\begin{equation}\label{eq:distance_bound}
  \big\|f^{n+1}-\rho^{n+1} M^{n+1}\big\|_{xv}^h
  \;\le\;
  \frac{M_{\max}^{n+1}}{\gamma_h}\,\Theta \Big( \varepsilon\,\|G^{n+1}\|_{xv}^h \;+\; \kappa_h\,M_{\max}^{n+1}\,\sqrt{L_v}\,\big\|\rho^{n+1}\big\|_x^h \Big).
\end{equation}
\end{theorem}

\begin{proof}
We focus our analysis on the L-step because it determines the final state $f^{n+1}$. Recall the L-step update \eqref{eq:L_fd_imex} projected onto the updated spatial basis ${X}_i^{n+1}$:
\[
  \frac{L_i^{n+1} - L_i^{(2)}}{\Delta t} = \langle \mathcal{A}_h(f^{(2)}), {X}_i^{n+1} \rangle_x^h + \frac{1}{\varepsilon} \langle \mathcal{L}_{h}^{n+1}(f^{n+1}), {X}_i^{n+1} \rangle_x^h.
\]
Multiplying by ${X}_i^{n+1}$ and summing over $i$, we reconstruct the full distribution update projected onto the spatial subspace. Let $P_X^{n+1}$ be the orthogonal projector onto $\mathrm{span}\{{X}_i^{n+1}\}$. Rearranging the update equation to isolate the Fokker--Planck term yields:
\begin{equation}\label{eq:G_derivation}
  P_X^{n+1}\mathcal{L}_{h}^{n+1}(f^{n+1}) = \varepsilon \underbrace{\left[ \frac{f^{n+1} - f^{(2)}}{\Delta t} - P_X^{n+1}\mathcal{A}_h(f^{(2)}) \right]}_{:= G^{n+1}}.
\end{equation}
In the assumed asymptotic regime, the term $G^{n+1}$ depends only on bounded transport and time-difference operators acting on the smooth solution; thus $\|G^{n+1}\|_{xv}^h = \mathcal{O}(1)$ as $\varepsilon \to 0$.

We now decompose the Fokker--Planck operator image into projected and orthogonal components:
\[
  \big\|\mathcal L_{h}^{n+1}(f^{n+1})\big\|_{xv}^h \le \big\|P_X^{n+1}\mathcal L_{h}^{n+1}(f^{n+1})\big\|_{xv}^h + \big\|(I-P_X^{n+1})\mathcal L_{h}^{n+1}(f^{n+1})\big\|_{xv}^h.
\]
From \eqref{eq:G_derivation}, the projected component is exactly $\varepsilon \|G^{n+1}\|_{xv}^h$. For the orthogonal component, since $f^{n+1}$ lies in the range of $P_X^{n+1}$, Proposition~\ref{prop:fluctuation_bound} bounds the orthogonal component by $\kappa_h \|f^{n+1}\|_{xv}^h$.

To utilize the coercivity estimate, we use the micro--macro decomposition $f^{n+1} = \rho^{n+1} M^{n+1} + g^{n+1}$. We bound the state norm by
\[
  \|f^{n+1}\|_{xv}^h \le \|\rho^{n+1}M^{n+1}\|_{xv}^h + \|g^{n+1}\|_{xv}^h \le M_{\max}^{n+1}\sqrt{L_v}\|\rho^{n+1}\|_x^h + \|g^{n+1}\|_{xv}^h.
\]

We now test the operator against the relative fluctuation $g^{n+1}/M^{n+1}$. Proposition~\ref{prop:coercivity} provides the lower bound $-\langle \mathcal L_{h}^{n+1}(f^{n+1}),\,g^{n+1}/M^{n+1} \rangle_{xv}^h \ge \gamma_h\,\|g^{n+1}/M^{n+1}\|_{xv}^{h,2}$. Conversely, decomposing $\mathcal L_{h}^{n+1}$ into $\varepsilon G^{n+1} + (I-P_X^{n+1})\mathcal L_{h}^{n+1}$ and testing against $g^{n+1}/M^{n+1}$ yields an upper bound:
\begin{align*}
  -\big\langle \mathcal L_{h}^{n+1}(f^{n+1}),\,g^{n+1}/M^{n+1} \big\rangle_{xv}^h
  &\le \big( \varepsilon \|G^{n+1}\|_{xv}^h + \kappa_h \|f^{n+1}\|_{xv}^h \big) \|g^{n+1}/M^{n+1}\|_{xv}^h \\
  &\le \Big( \varepsilon \|G^{n+1}\|_{xv}^h + \kappa_h \big[ M_{\max}^{n+1}\sqrt{L_v}\|\rho^{n+1}\|_x^h + \|g^{n+1}\|_{xv}^h \big] \Big) \|g^{n+1}/M^{n+1}\|_{xv}^h.
\end{align*}
Using the inequality $\|g^{n+1}\|_{xv}^h = \|M^{n+1} (g^{n+1}/M^{n+1})\|_{xv}^h \le M_{\max}^{n+1} \|g^{n+1}/M^{n+1}\|_{xv}^h$, we can absorb the $\|g^{n+1}\|$ term. Assuming $\|g^{n+1}/M^{n+1}\|_{xv}^h > 0$, we divide by $\|g^{n+1}/M^{n+1}\|_{xv}^h$ and rearrange terms:
\[
  (\gamma_h - \kappa_h M_{\max}^{n+1}) \|g^{n+1}/M^{n+1}\|_{xv}^h \le \varepsilon \|G^{n+1}\|_{xv}^h + \kappa_h M_{\max}^{n+1} \sqrt{L_v} \|\rho^{n+1}\|_x^h.
\]
Identifying the amplification factor $\Theta = \gamma_h / (\gamma_h - \kappa_h M_{\max}^{n+1})$, we obtain the bound for $\|g^{n+1}/M^{n+1}\|_{xv}^h$. The distance to equilibrium is simply $\|f^{n+1}-\rho^{n+1}M^{n+1}\|_{xv}^h = \|g^{n+1}\|_{xv}^h \le M_{\max}^{n+1}\|g^{n+1}/M^{n+1}\|_{xv}^h$, which yields \eqref{eq:distance_bound}. The residual bound \eqref{eq:residual_bound} follows by substituting this back into the norm decomposition of $\mathcal{L}_{h}$.
\end{proof}

\begin{remark}
If the electric field is spatially constant (even if large), we may choose $\bar{\beta}=\beta_p$, implying $\delta \beta \equiv 0$ and $\kappa_h = 0$. In this case, the error is purely $\mathcal{O}(\varepsilon)$, recovering the strong AP property.  More generally, if the field fluctuation scales as $\kappa_h = \mathcal{O}(\varepsilon)$ for bounded density at the fixed mesh, then both terms in~\eqref{eq:distance_bound} are $\mathcal{O}(\varepsilon)$ and the strong AP property is recovered.  For a general $\kappa_h$, Theorem~\ref{thm:main_AP} gives a relaxation estimate of order $\mathcal{O}(\varepsilon)+\mathcal{O}(\kappa_h)$.
\end{remark}

\begin{remark}
The analysis in this section assumes the electric field $E^{n+1}$ is given and fixed during the kinetic update. In the full VPFP system, $E$ is coupled to $\rho$ via the Poisson equation. A complete AP analysis of the fully nonlinear coupled system would require accounting for errors in the macroscopic field solver, which is beyond the scope of this work.
\end{remark}

\begin{remark}
Although Theorem \ref{thm:main_AP} only proves AP property up to the spatial fluctuation  $\mathcal{O}(\kappa_h)$, we observe in the numerical experiments that the first-order scheme exhibits an AP error of $\mathcal{O}(\varepsilon)$, regardless of the field fluctuation.

\end{remark}

\section{Numerical experiments}
\label{sec:numerical_experiments}
We validate the accuracy, efficiency, and AP property of the proposed first-order and second-order low-rank IMEX schemes. All numerical tests are performed in a 1D1V phase space with periodic boundary condition in $x$ and zero flux boundary condition in the truncated velocity domain $[-6,6]$.


\subsection{Convergence and AP tests}
\label{sec:conv_ap_tests}

We consider two types of initial data: 

\medskip\noindent
\emph{Non-equilibrium initial data (used for the first-order scheme).}
We take a shifted Gaussian in $v$ with a spatially varying density,
\[
  f_{\mathrm{neq}}^0(x,v)
  =\frac{\rho^0(x)}{\sqrt{2\pi}}\exp\!\Big(-\frac{(v-1.5)^2}{2}\Big),
  \qquad
  \rho^0(x)=\sqrt{2\pi}\,\big(2+\cos(2\pi x)\big),\quad x\in[0,1],
\]
and the background charge
\[
  \eta(x)=\frac{2\sqrt{2\pi}}{I_0(1)}\,e^{\cos(2\pi x)},
\]
so that the neutrality condition $\int_0^1(\rho^0-\eta)\,\dx=0$ holds. Here $I_0$ denotes the modified Bessel function of the first kind of order zero, and $I_0(1)\approx 1.2661$. 
The initial electric field $E^0$ is obtained from the Poisson equation
\[
  -\partial_{xx}\phi^0=\rho^0(x)-\eta(x),\qquad E^0(x)=-\partial_x\phi^0.
\]

\medskip\noindent
\emph{Equilibrium initial data (used for the second-order scheme).}
For the second-order scheme, the AP property in the full-tensor framework is guaranteed when the initial data is close to the equilibrium.
Accordingly, for the second-order low-rank scheme, we consider equilibrium initial data of the form
\[
  f_{\mathrm{eq}}^0(x,v) = \frac{\rho^0(x)}{\sqrt{2\pi}}\exp\!\Big(-\frac{(v-E^0(x))^2}{2}\Big),
\]
where $\rho^0(x)$ and $E^0(x)$ are the same as above.


\medskip
Let $\|\cdot\|_{L^1_{x,v}}$ denote the discrete $L^1$ norm in phase space induced by the $(x,v)$ grid and $\|\cdot\|_{L^1_v}$ the corresponding $L^1$ norm in $v$ at fixed $x$.
To measure the convergence rate in time, we consider a sequence of time steps
\[
  \Delta t_0 > \Delta t_1 > \cdots > \Delta t_K,
  \qquad \Delta t_{k+1} = \frac{\Delta t_k}{2},
\]
and for each $k$ we evolve the scheme up to a fixed final time $T$ with time step $\Delta t_k$.
The time-discretization error at level $k$ is defined by \modifysecond{the difference, at the same final time $T$, between two complete simulations carried out independently with two successive time-step sizes $\Delta t_k$ and $\Delta t_{k+1}=\Delta t_k/2$,}
\begin{equation}\label{eq:dt_error_successive}
  \mathcal E_{\Delta t_k}(T)
  := \big\| f^{\Delta t_k}(T) - f^{\Delta t_{k+1}}(T)\big\|_{L^1_{x,v}}.
\end{equation}
If the scheme is $p$th-order accurate in time, then for small $\Delta t_k$ one expects $\mathcal E_{\Delta t_k}(T)\approx C\,(\Delta t_k)^p$, so the slope of $\mathcal E_{\Delta t_k}(T)$ versus $\Delta t_k$ on a log--log plot reflects the observed convergence rate $p$.

To assess AP property, we quantify the deviation of $f$ from its local Maxwellian $\rho(x,t)M(x,v,t)$ and monitor the global discrete $L^1_{x,v}$ error
\begin{equation}\label{eq:AP_error_global}
  \mathcal E_{AP}(t)
  := \big\| f(\cdot,\cdot,t) - \rho(\cdot,t)M(\cdot,\cdot,t)\big\|_{L^1_{x,v}}.
\end{equation}

\paragraph{First-order scheme.}
For the first-order low-rank IMEX scheme, we use $(N_x,N_v)=(64,64)$ and rank $r=10$.
We consider both the kinetic regime $\varepsilon=1$ and the high-field regime $\varepsilon=10^{-6}$. 
For the time-step convergence study we take
\[
  \Delta t_0 = 10^{-3},\quad
  \Delta t_{k+1} = \frac{\Delta t_k}{2},\quad k=0,\dots,4,
  \qquad T_{\max}=5\times 10^{-3}.
\]
Figure~\ref{fig:dt-ap-first}(a) shows $\mathcal E_{\Delta t_k}(T_{\max})$ versus $\Delta t_k$ on a log--log scale. 
The observed slopes are close to one for both $\varepsilon=1$ and $\varepsilon=10^{-6}$, indicating first-order convergence in time that is essentially uniform with respect to $\varepsilon$.

For the AP test we fix $\Delta t=2.5\times 10^{-3}$, $T_{\max}=0.1$, and $r=10$, and monitor the global AP error $\mathcal E_{AP}(t)$ defined in \eqref{eq:AP_error_global}.
As shown in Figure~\ref{fig:dt-ap-first}(b), the error decays rapidly over a few time steps and then saturates at a level proportional to $\varepsilon$.
This behavior is consistent with the AP property of the first-order full-tensor scheme, indicating that the low-rank formulation retains robustness to non-equilibrium initial data.

\paragraph{Second-order scheme.}
For the second-order low-rank IMEX scheme, we again use $(N_x,N_v)=(64,64)$.
When $\varepsilon=O(1)$, there is no guarantee that the solution is close to an equilibrium manifold. To separate time-discretization error from rank-truncation error in the convergence study, we therefore use a slightly larger rank $r=15$ for $\varepsilon=1$.
When $\varepsilon$ is very small, the solution is expected to remain close to a local Maxwellian; in this case we choose a smaller rank $r=6$ for $\varepsilon=10^{-6}$ to illustrate the efficiency of the low-rank approach in the fluid regime. The time-step convergence study uses
\[
  \Delta t_0 = 5\times 10^{-5},\quad
  \Delta t_{k+1} = \frac{\Delta t_k}{2},\quad k=0,\dots,4,
  \qquad T_{\max}=5\times 10^{-4}.
\]
Figure~\ref{fig:dt-ap-first}(c) reports $\mathcal E_{\Delta t_k}(T_{\max})$ versus $\Delta t_k$. 
The measured slopes are close to two for both values of $\varepsilon$, confirming second-order convergence in time and suggesting that the order is essentially uniform in the kinetic and fluid regimes.

For the AP test we fix $\Delta t=5\times 10^{-5}$ and $T_{\max}=2\times 10^{-3}$, and monitor $\mathcal E_{AP}(t)$.
The results in Figure~\ref{fig:dt-ap-first}(d) demonstrate rapid relaxation towards the local Maxwellian with a plateau at a level proportional to $\varepsilon$, consistent with the AP theory for the second-order full-tensor scheme. 

\begin{figure}[htbp]
  \centering
  \begin{subfigure}[t]{0.48\textwidth}
    \centering
    \includegraphics[width=\linewidth]{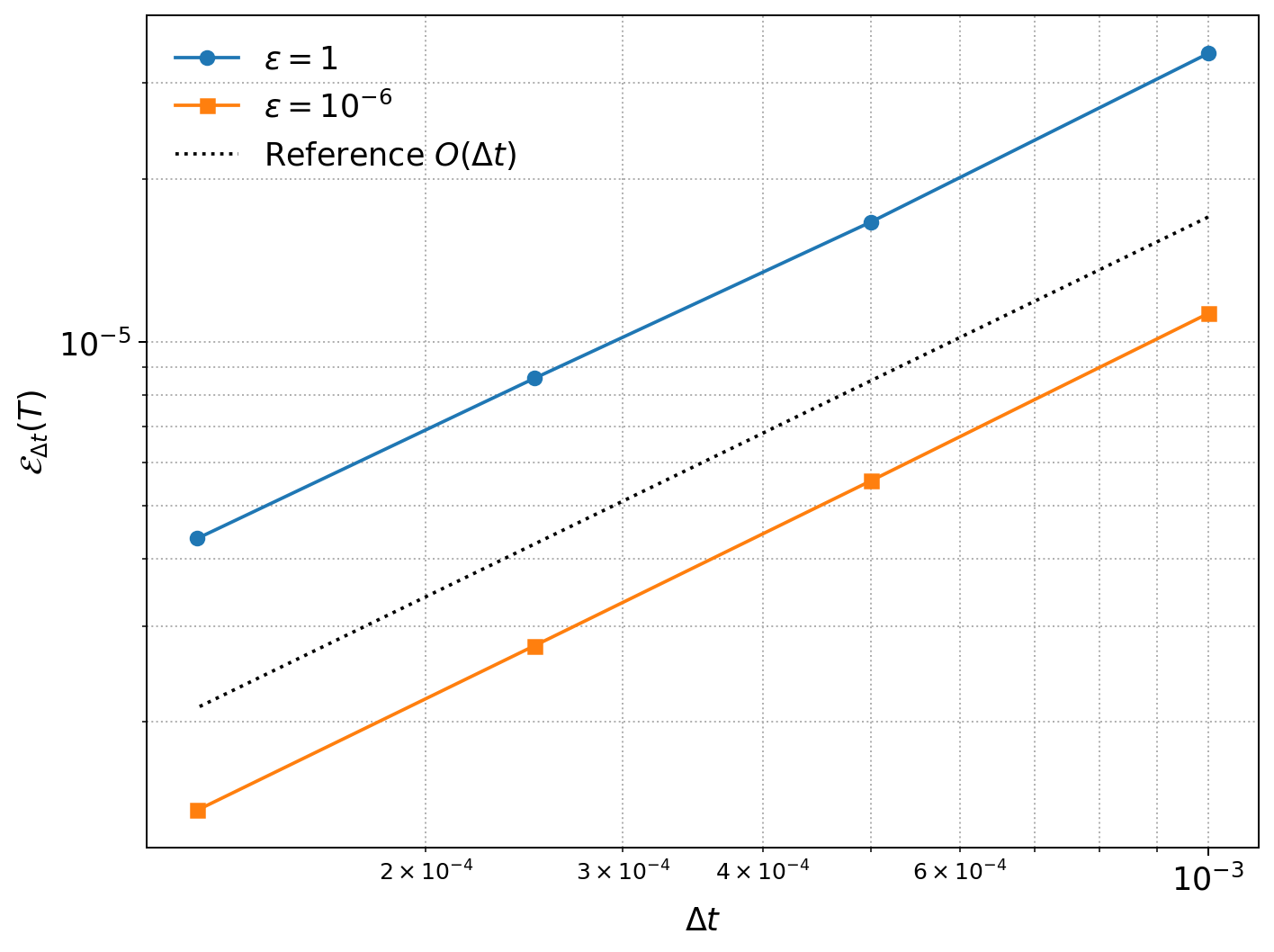}
    \caption{First-order: convergence of $f$ vs.\ $\Delta t$ for $\varepsilon=1$ and $\varepsilon=10^{-6}$.}
    \label{fig:dt-conv-first}
  \end{subfigure}\hfill
  \begin{subfigure}[t]{0.48\textwidth}
    \centering
    \includegraphics[width=\linewidth]{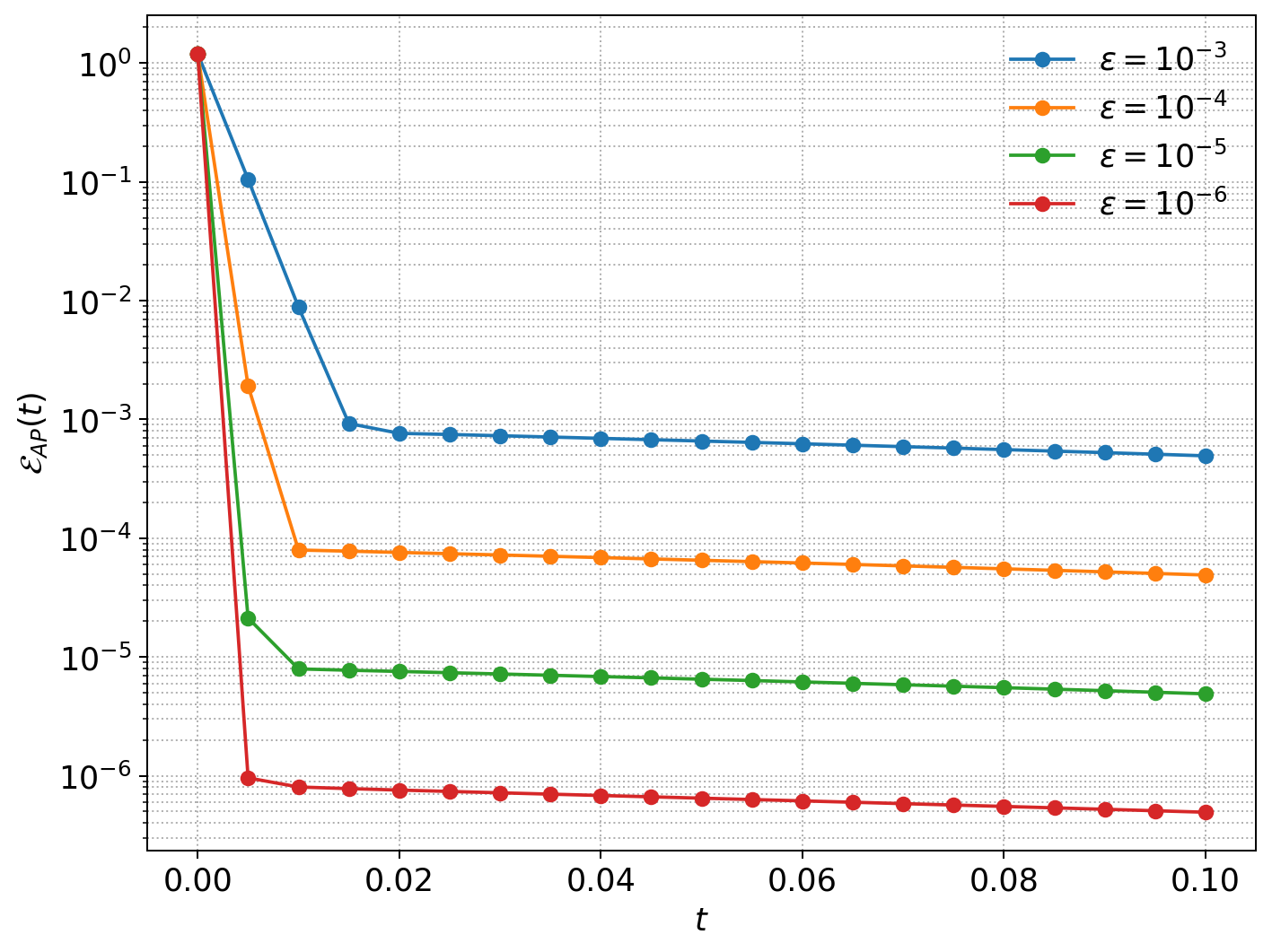}
    \caption{First-order: global AP error $\mathcal E_{AP}(t)$ over time.}
    \label{fig:ap-error-first}
  \end{subfigure}

  \vspace{0.6em}

  \begin{subfigure}[t]{0.48\textwidth}
    \centering
    \includegraphics[width=\linewidth]{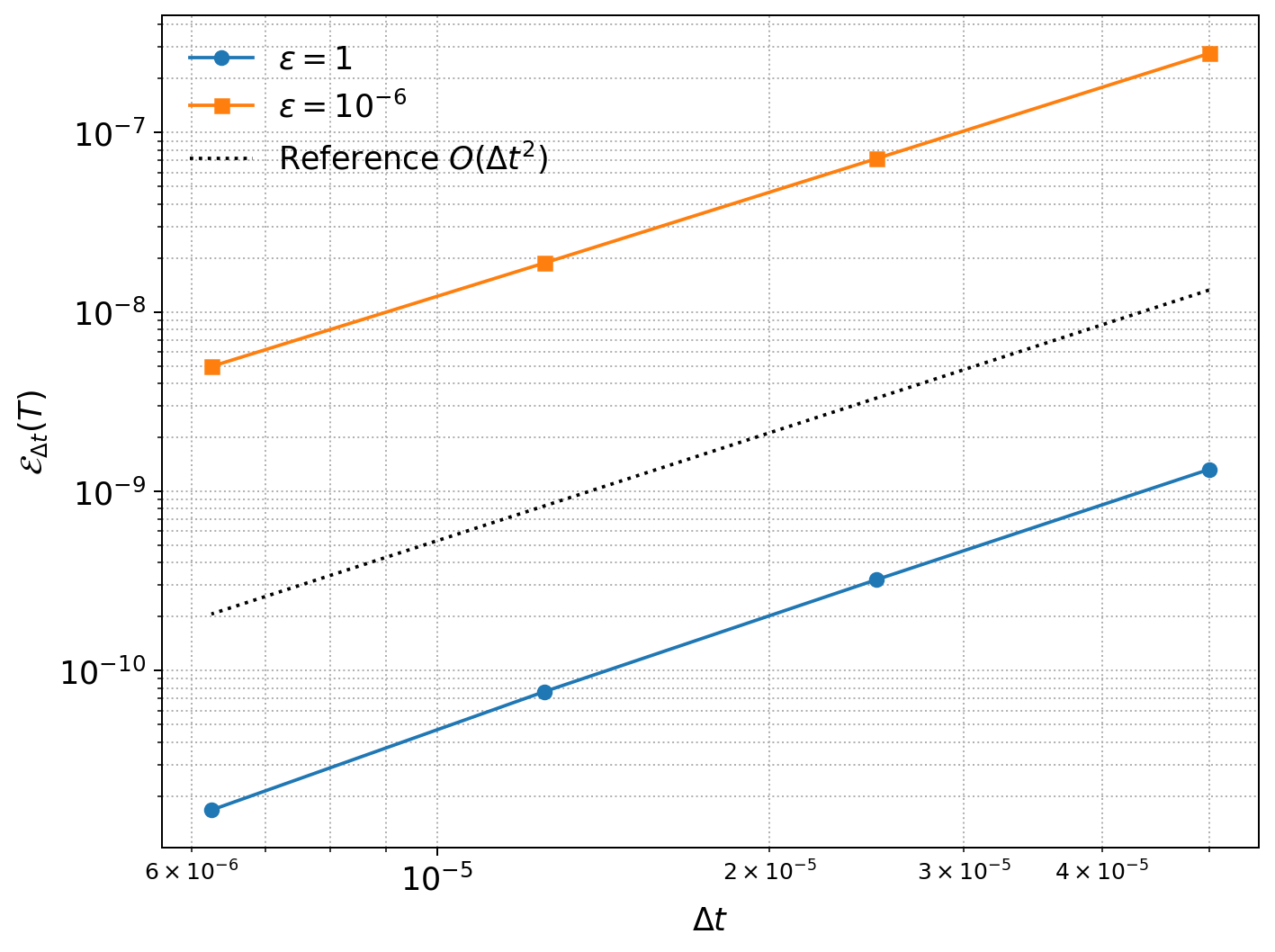}
    \caption{Second-order: convergence of $f$ vs.\ $\Delta t$ for $\varepsilon=1$ and $\varepsilon=10^{-6}$.}
    \label{fig:dt-conv-second}
  \end{subfigure}\hfill
  \begin{subfigure}[t]{0.48\textwidth}
    \centering
    \includegraphics[width=\linewidth]{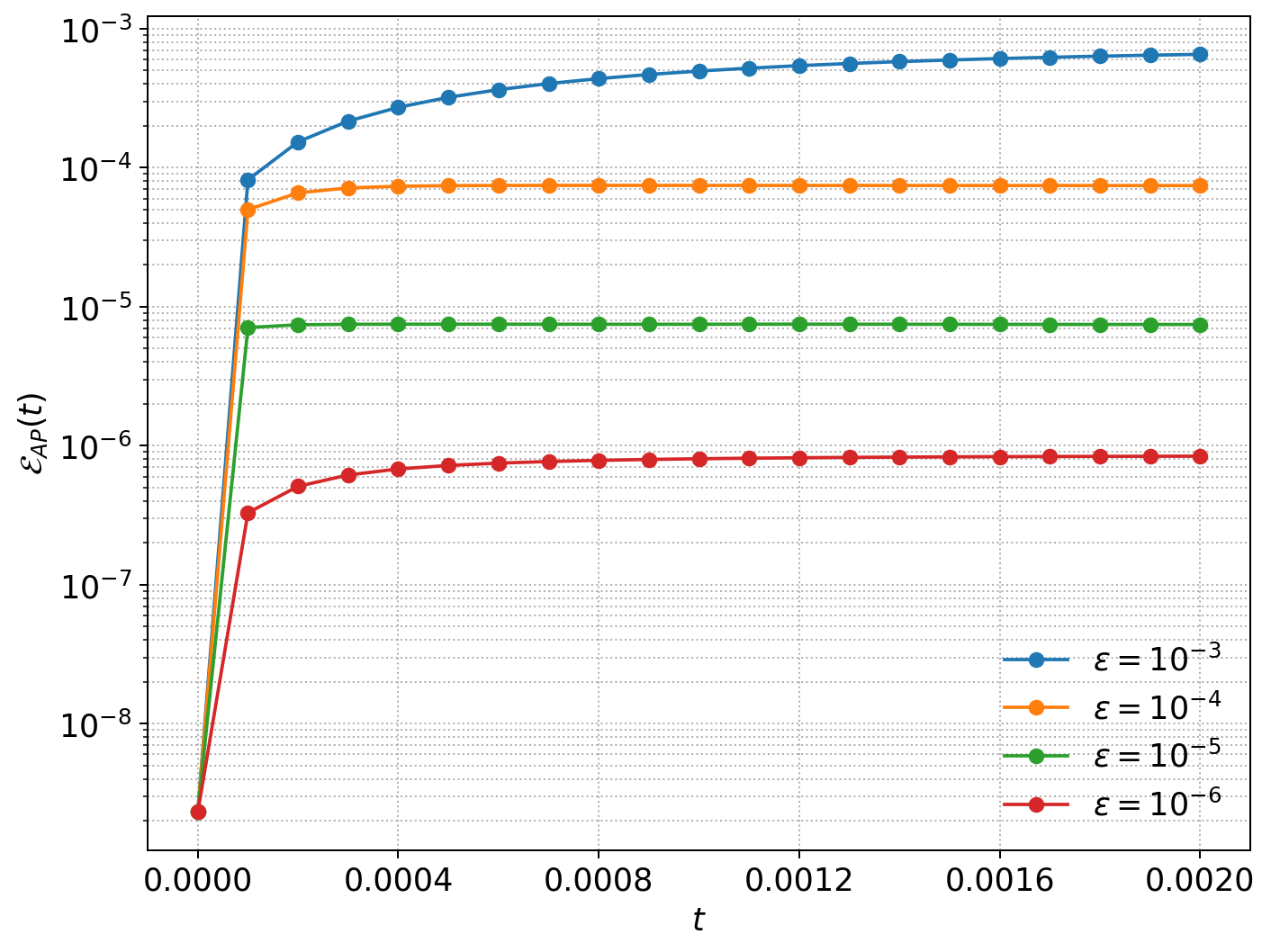}
    \caption{Second-order: global AP error $\mathcal E_{AP}(t)$ over time.}
    \label{fig:ap-error-second}
  \end{subfigure}
  \caption{Time-step convergence and AP property for the first-order (top) and second-order (bottom) low-rank IMEX schemes.}
  \label{fig:dt-ap-first}
\end{figure}

\subsection{Mixed regime}
\label{sec:mixed_regime}

In the rest of the tests, we use the first-order low-rank IMEX scheme.  

We prescribe a spatially varying Knudsen number
\[
\varepsilon(x)=
\begin{cases}
\varepsilon_0+\tfrac12\big(\tanh(5-10x)+\tanh(5+10x)\big), & x\le 0.3,\\[2pt]
\varepsilon_0, & x>0.3,
\end{cases}
\qquad \varepsilon_0=10^{-3},\quad x\in[-1,1],
\]
so that the left part of the domain has $\varepsilon(x)$ of order one, while the right part has a smaller, nearly constant value $\varepsilon(x)\approx \varepsilon_0$. 
The initial density and distribution are
\[
\rho^0(x)=\frac{\sqrt{2\pi}}{6}\big(2+\sin(\pi x)\big),\qquad
f^0(x,v)=\frac{\rho^0(x)}{\sqrt{2\pi}}\exp\!\Big(-\frac{(v-E^0(x))^2}{2}\Big),
\]
with background charge
\[
  \eta(x) = C_\eta\,e^{\cos(\pi x)},
\]
where the constant $C_\eta$ is chosen so that the neutrality condition
\[
  \int_{-1}^1\big(\eta(x)-\rho^0(x)\big)\,\dx = 0
\]
holds (numerically, $C_\eta \approx 0.6599$).
Thus $f^0$ is a local Maxwellian with density $\rho^0(x)$ and field $E^0(x)$. 
We use $N_x=100$,  $N_v=128$, a time step $\Delta t=10^{-4}$, and a moderate fixed rank $r = 13$ for the low-rank approximation.

Figure~\ref{fig:mixed-density-efield} illustrates the density and electric field at $t=0.10$ and $t=0.30$, demonstrating that the low-rank scheme accurately reproduces the macroscopic evolution in both regions. Figure~\ref{fig:mixed-phases} compares the phase portraits of the full-tensor reference solution and the low-rank solution, alongside their absolute difference at the final time. The low-rank scheme exhibits excellent agreement with the full-tensor benchmark. These results confirm that a moderate rank is sufficient to effectively capture the dominant phase-space structures in both the order-one and small-$\varepsilon(x)$ subdomains.
\begin{figure}[htbp]
  \centering
  \begin{subfigure}[t]{0.48\textwidth}
    \centering
    \includegraphics[width=\linewidth]{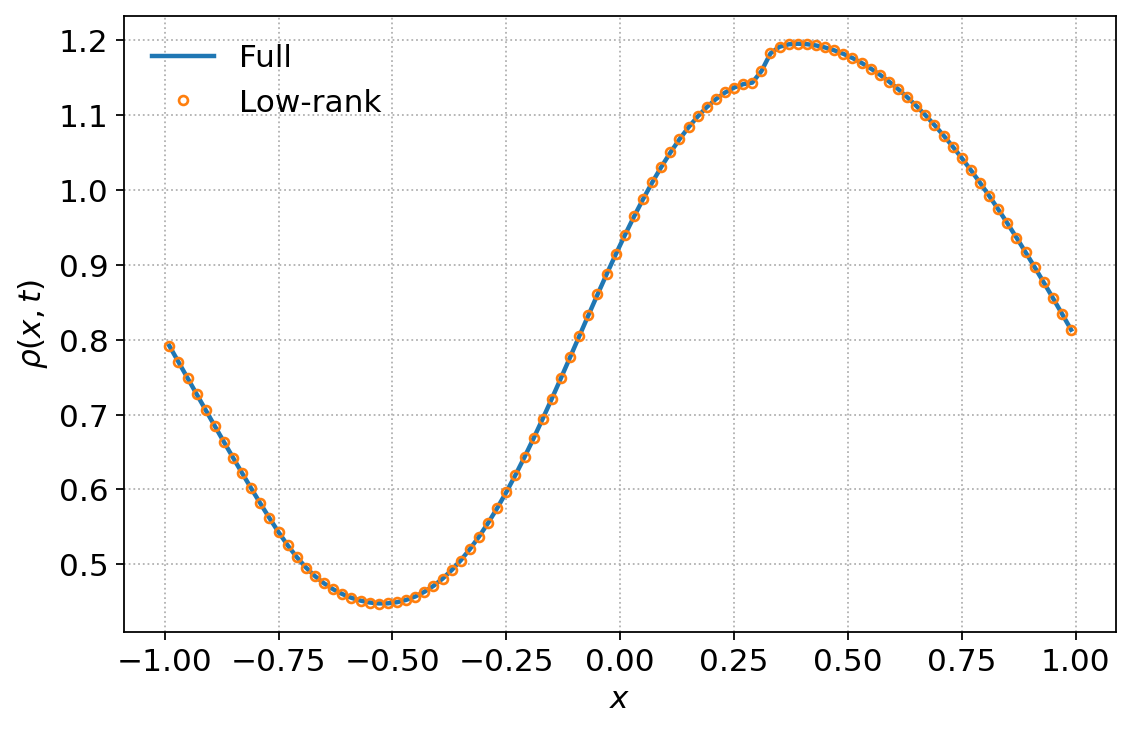}
    \caption{Density at $t=0.10$.}
    \label{fig:mixed-density-010}
  \end{subfigure}\hfill
  \begin{subfigure}[t]{0.48\textwidth}
    \centering
    \includegraphics[width=\linewidth]{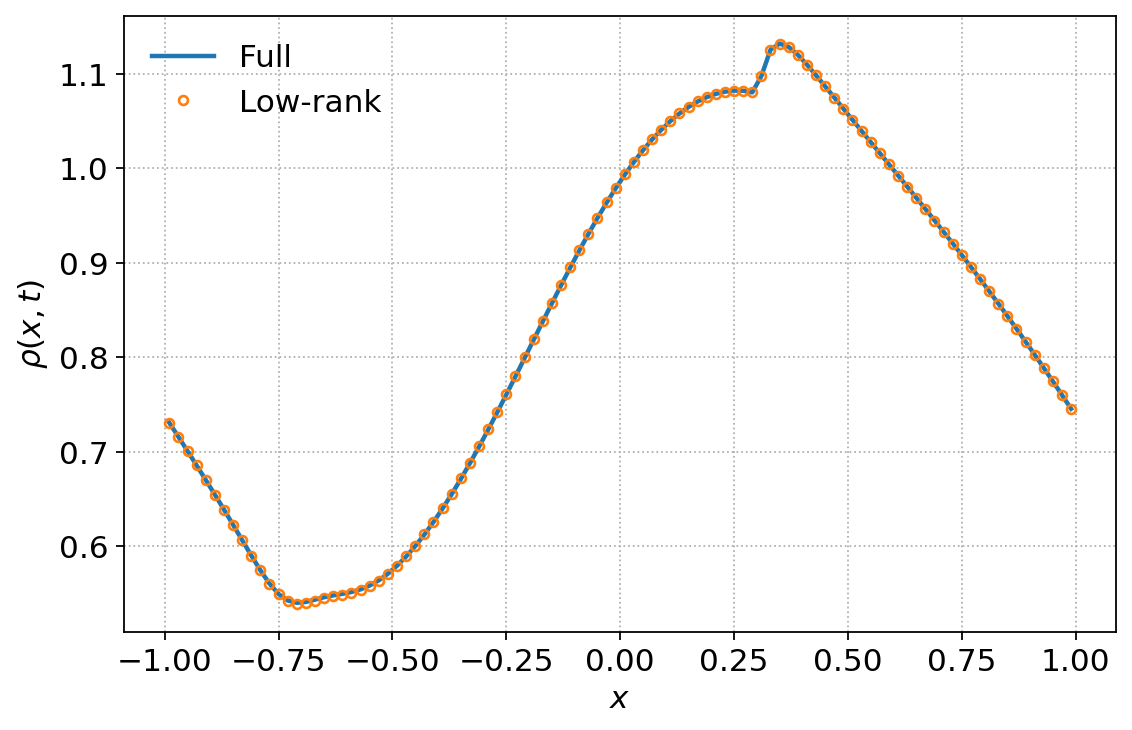}
    \caption{Density at $t=0.30$.}
    \label{fig:mixed-density-030}
  \end{subfigure}

  \vspace{0.6em}

  \begin{subfigure}[t]{0.48\textwidth}
    \centering
    \includegraphics[width=\linewidth]{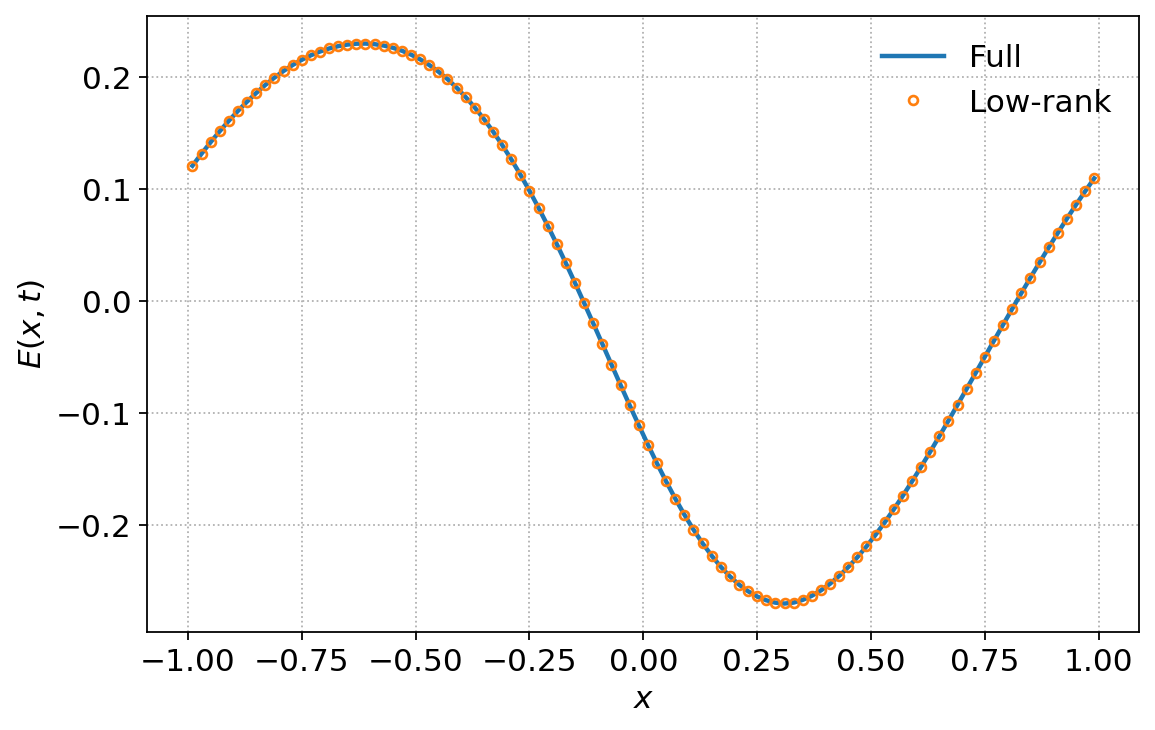}
    \caption{Electric field at $t=0.10$.}
    \label{fig:mixed-efield-010}
  \end{subfigure}\hfill
  \begin{subfigure}[t]{0.48\textwidth}
    \centering
    \includegraphics[width=\linewidth]{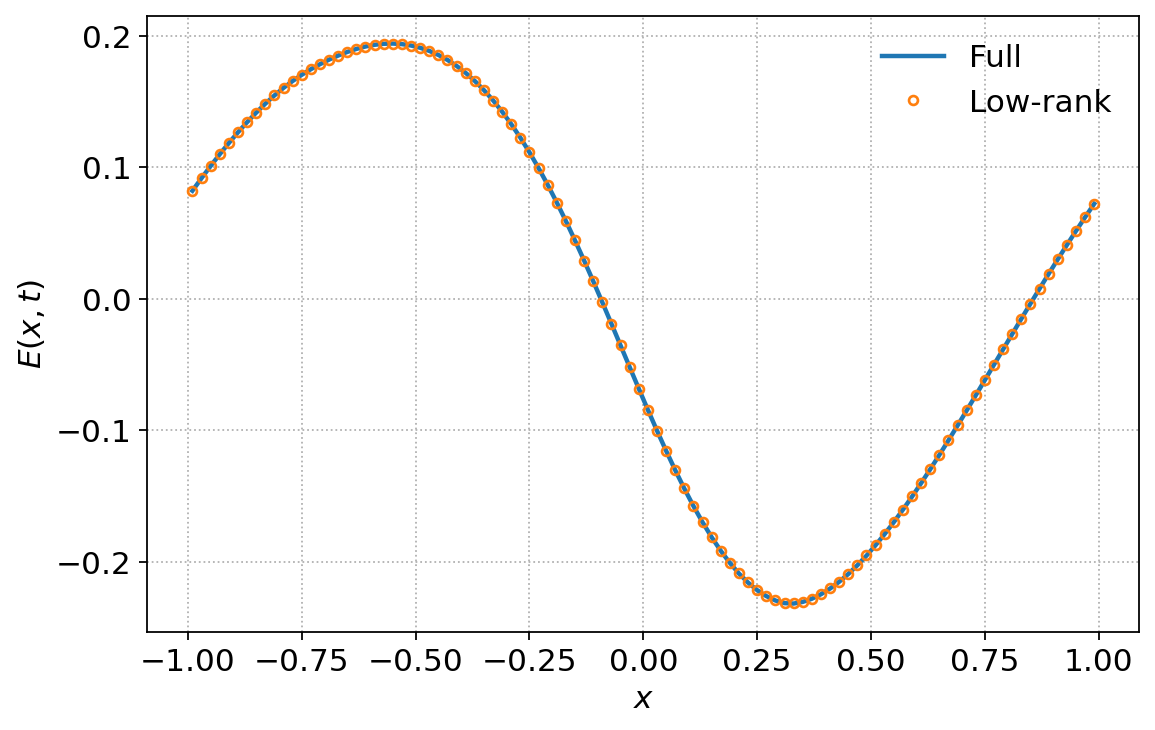}
    \caption{Electric field at $t=0.30$.}
    \label{fig:mixed-efield-030}
  \end{subfigure}
  \caption{Mixed-regime test: comparison of macroscopic density (top row) and electric field (bottom row) between the full-tensor reference solution (solid lines) and the first-order low-rank IMEX scheme (circles) at $t=0.10$ and $t=0.30$.}
  \label{fig:mixed-density-efield}
\end{figure}

\begin{figure}[htbp]
  \centering
  \begin{subfigure}[t]{0.32\textwidth}
    \centering
    \includegraphics[width=\linewidth]{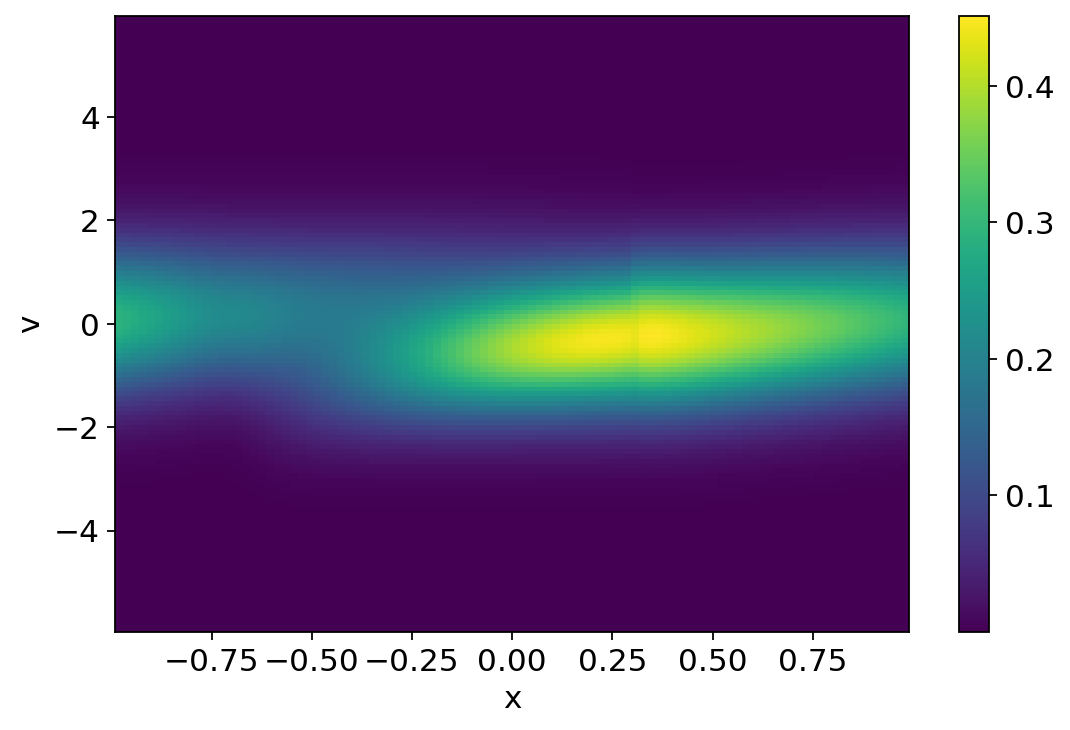}
    \caption{Full model.}
    \label{fig:phase-full}
  \end{subfigure}\hfill
  \begin{subfigure}[t]{0.32\textwidth}
    \centering
    \includegraphics[width=\linewidth]{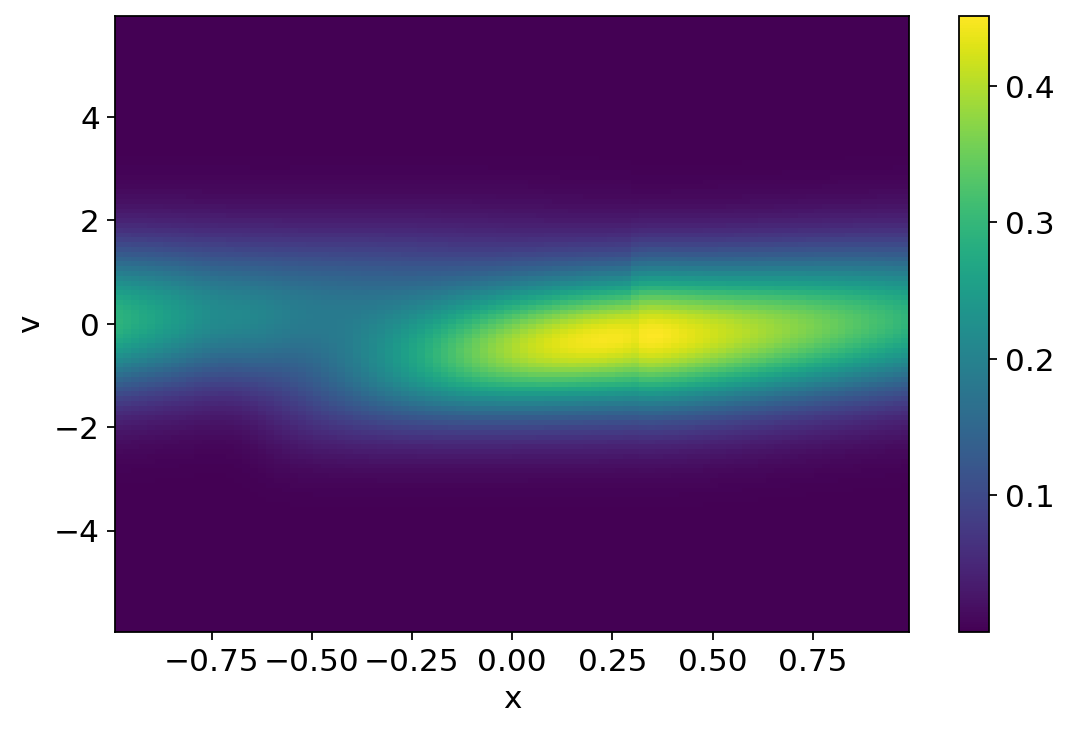}
    \caption{Low-rank model.}
    \label{fig:phase-lowrank}
  \end{subfigure}\hfill
  \begin{subfigure}[t]{0.33\textwidth}
    \centering
    \includegraphics[width=\linewidth]{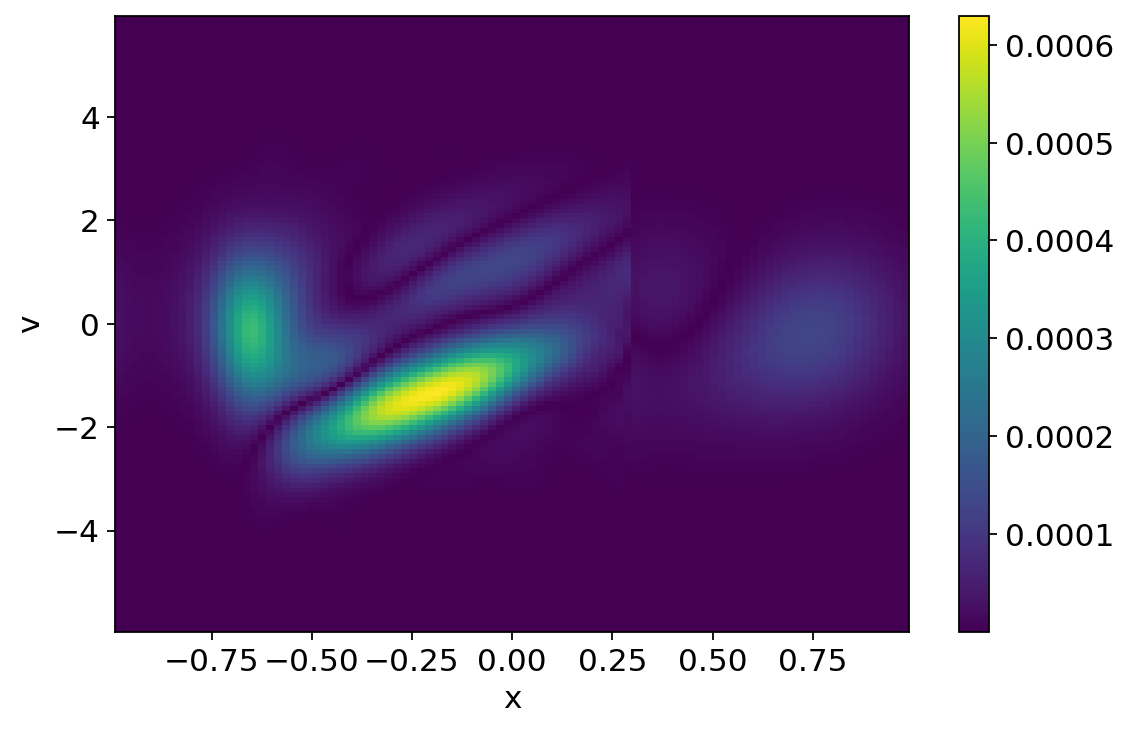}
    \caption{$|f_\text{full}-f_\text{low}|$.}
    \label{fig:phase-absdiff}
  \end{subfigure}
  \caption{Phase plots in the mixed regime test: full tensor, low-rank, and absolute difference.}
  \label{fig:mixed-phases}
\end{figure}

Finally, Figure~\ref{fig:mixed-eps-and-errors} reports the spatial profile of $\varepsilon(x)$ together with the pointwise AP error
\[
  \mathcal E_{AP}(x,t) = \|f(x,\cdot,t) - \rho(x,t)M(x,\cdot,t)\|_{L^1_v}
\]
at $t=0.10$ and $t=0.30$. The error is of order $\mathcal{O}(\varepsilon(x))$, consistent with the expected AP behavior and with the mixed-regime results for the full-tensor reference scheme.

\begin{figure}[htbp]
  \centering
  \begin{subfigure}[t]{0.33\textwidth}
    \centering
    \includegraphics[width=\linewidth]{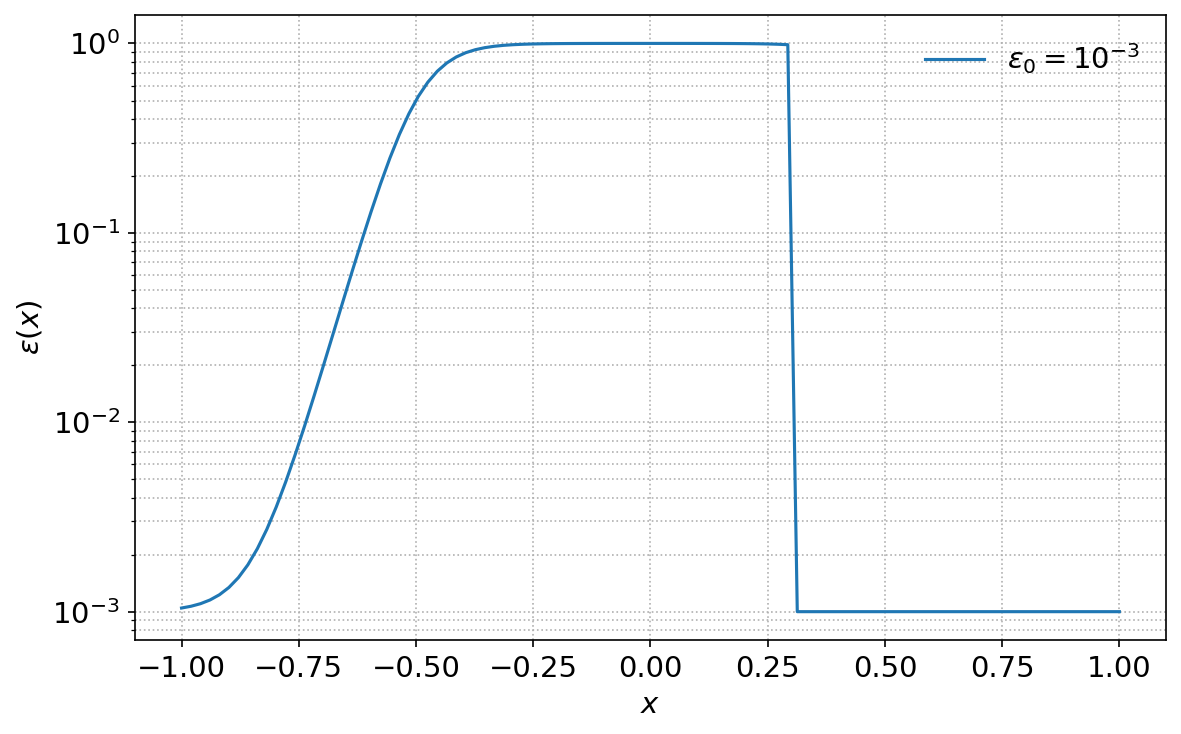}
    \caption{$\varepsilon(x)$ profile.}
    \label{fig:eps-profile}
  \end{subfigure}\hfill
  \begin{subfigure}[t]{0.32\textwidth}
    \centering
    \includegraphics[width=\linewidth]{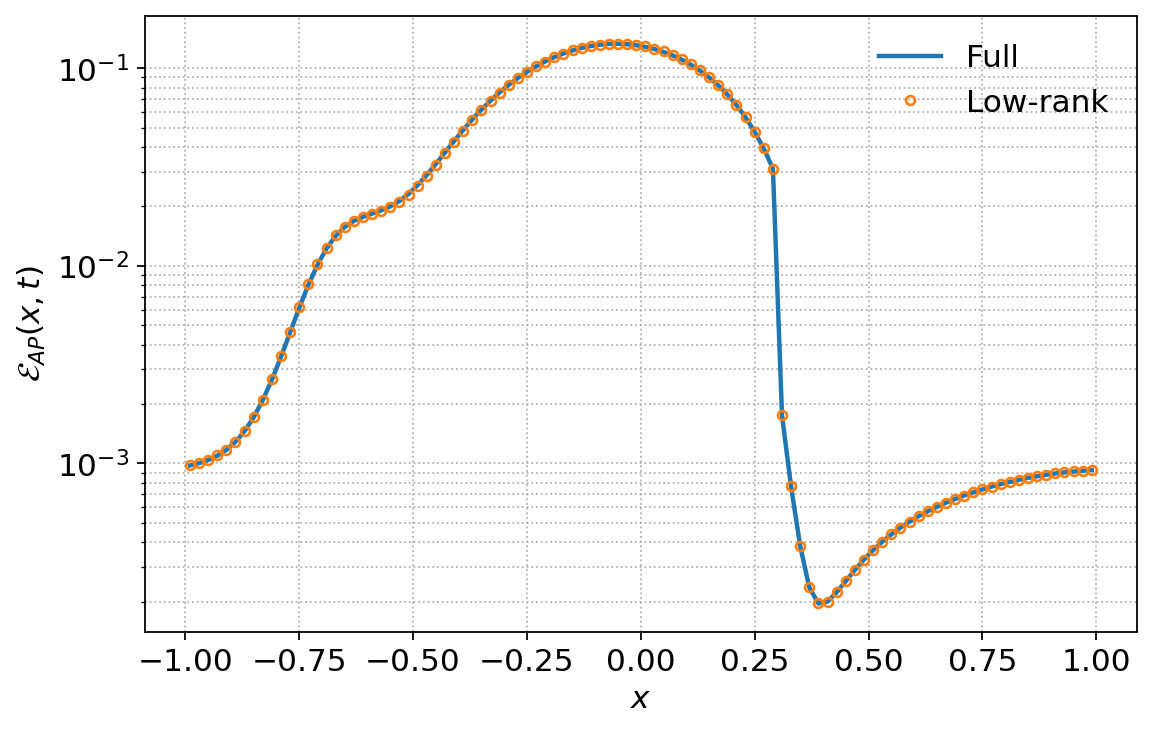}
    \caption{$\mathcal{E}_{AP}(x,t)$ at $t=0.10$.}
    \label{fig:l1v-010}
  \end{subfigure}\hfill
  \begin{subfigure}[t]{0.32\textwidth}
    \centering
    \includegraphics[width=\linewidth]{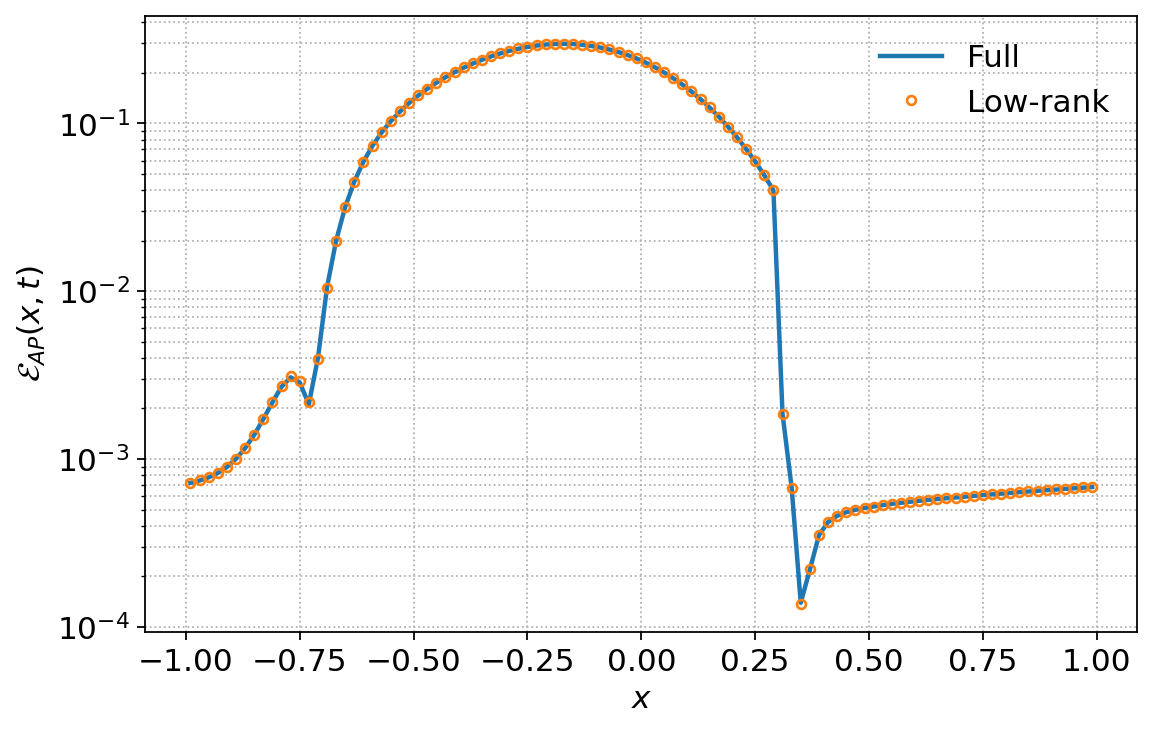}
    \caption{$\mathcal{E}_{AP}(x,t)$ at $t=0.30$.}
    \label{fig:l1v-030}
  \end{subfigure}
  \caption{Spatial profiles for the mixed-regime test: (a) the Knudsen number $\varepsilon(x)$; (b) and (c) the pointwise AP error $\mathcal{E}_{AP}(x,t)$ at $t=0.10$ and $t=0.30$, respectively.}
  \label{fig:mixed-eps-and-errors}
\end{figure}

\subsection{Bump-on-tail instability}
\label{sec:bump_on_tail}
The bump-on-tail instability (adapted from \cite{coughlin2022efficient}) serves as a nonlinear kinetic benchmark. The initial distribution consists of a thermal background and a cold shifted beam:
\[
  f^0(x,v)=\frac{\rho^0(x)}{\sqrt{2\pi}}
  \left(e^{-\tfrac{v^2}{2}}+e^{-\tfrac{(v-1.5)^2}{2T_{\rm cold}}}\right),
  \qquad T_{\rm cold}=5\times10^{-3},
\]
with spatially localized density and background charge
\[
  \rho^0(x)=0.3+\exp\!\Big(-\frac{(x-0.3)^2}{0.01}\Big),\qquad
  \eta(x)=0.3+\exp\!\Big(-\frac{(x-0.6)^2}{0.01}\Big),
\]
The Knudsen number is uniform: $\varepsilon(x)\equiv\varepsilon_0$. We take $N_x=100$, $N_v=128$, time step $\Delta t=10^{-3}$, and rank $r=6$. 
We examine two regimes:
\begin{itemize}\setlength{\itemsep}{2pt}
  \item \textbf{Fluid regime:} $\varepsilon_0=10^{-6}$, final time $T=0.5$.
  \item \textbf{Kinetic regime:} $\varepsilon_0=1$, final time $T=0.1$.
\end{itemize}
The rank $r=6$ is deliberately small relative to $(N_x,N_v)$ in order to demonstrate that the first-order low-rank IMEX scheme can capture the main bump-on-tail dynamics with only a few modes.
We note that Coughlin and Hu \cite{coughlin2022efficient} simulated this same benchmark using a quotient-based DLR method for $f/M$; while their approach is optimally efficient ($r \approx 3, 4$) in the fluid limit, it requires significantly higher ranks ($r \approx 20$) to capture the non-equilibrium dynamics in the kinetic regime. In contrast, our proposed method maintains high accuracy with a consistently small rank ($r=6$) across both regimes.

Figure~\ref{fig:tb-density-efield} displays the macroscopic density and electric field. 
The first-order low-rank IMEX scheme yields stable results that are consistent with the full-tensor reference solutions in both kinetic and fluid regimes. This confirms that the method effectively captures the correct macroscopic dynamics in these scenarios.

\begin{figure}[htbp]
  \centering
  \begin{subfigure}[t]{0.48\textwidth}
    \centering
    \includegraphics[width=\linewidth]{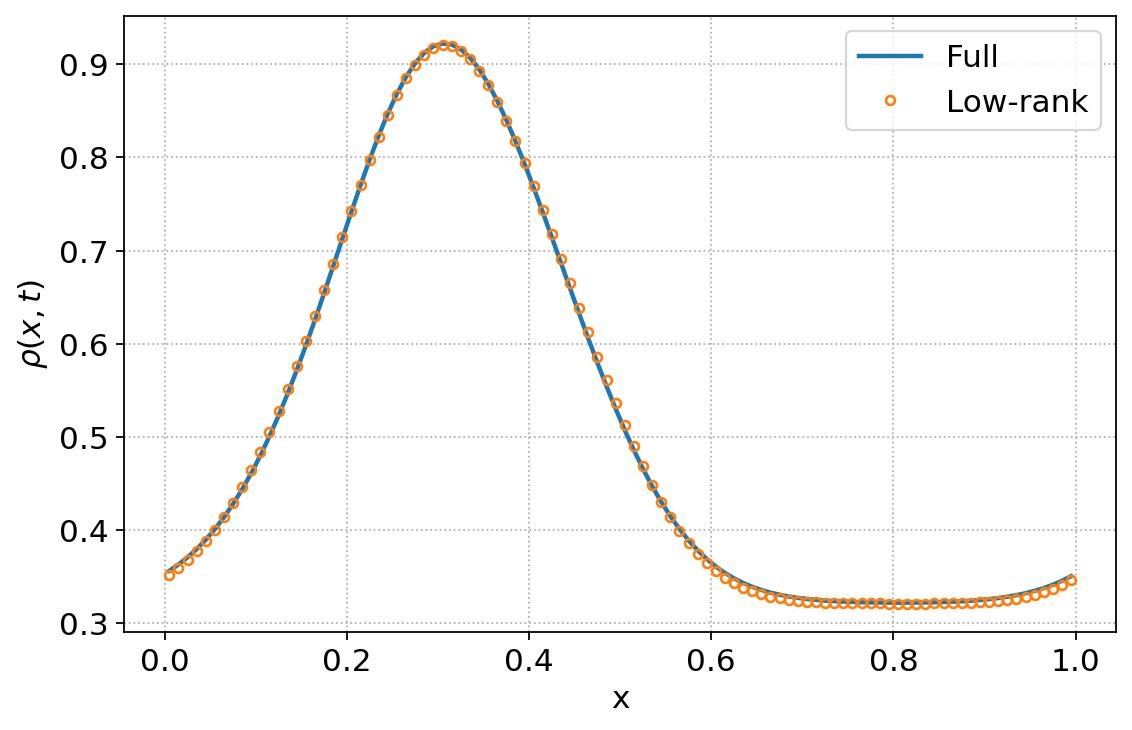}
    \caption{Density at $t=0.10$, $\varepsilon_0=1$.}
    \label{fig:tb-density-010}
  \end{subfigure}\hfill
  \begin{subfigure}[t]{0.48\textwidth}
    \centering
    \includegraphics[width=\linewidth]{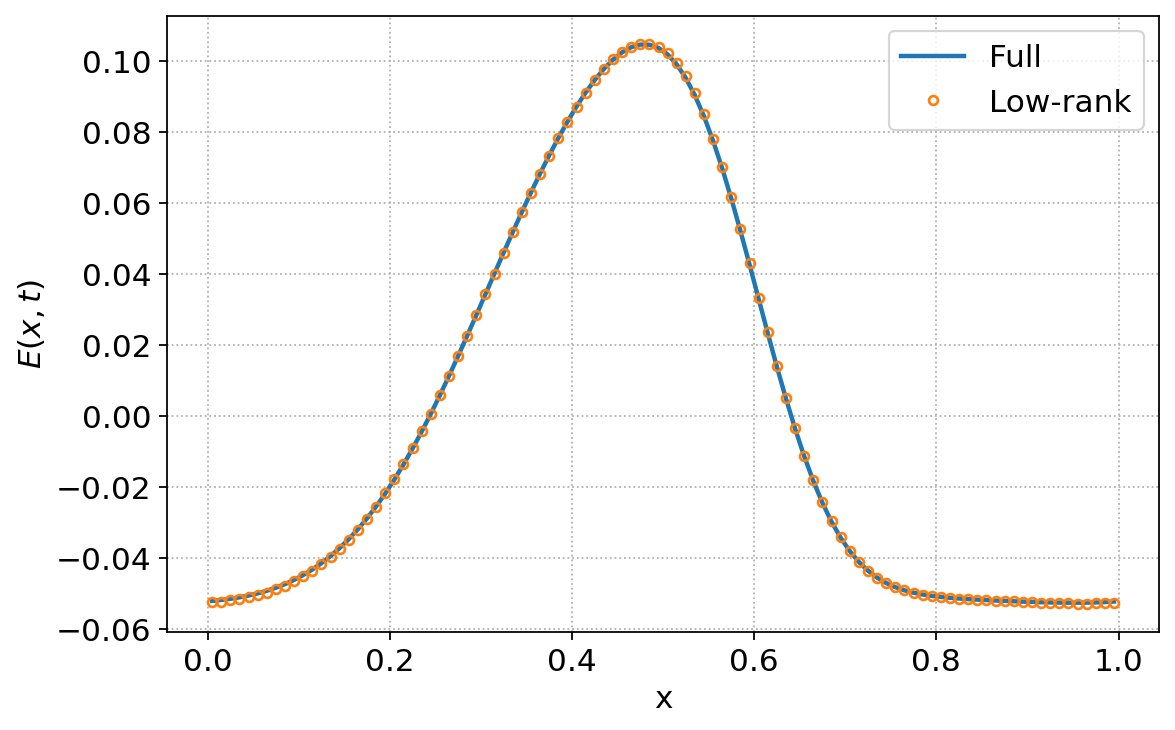}
    \caption{Electric field at $t=0.10$, $\varepsilon_0=1$.}
    \label{fig:tb-efield-010}
  \end{subfigure}

  \vspace{0.6em}
  \begin{subfigure}[t]{0.48\textwidth}
    \centering
    \includegraphics[width=\linewidth]{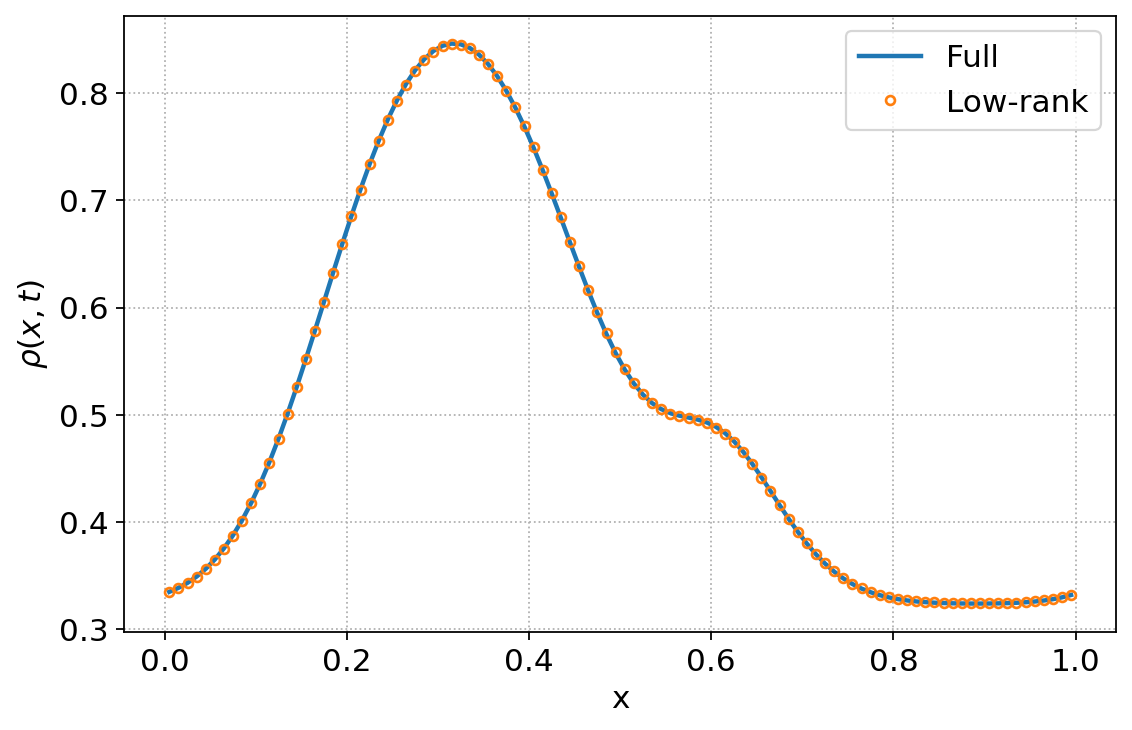}
    \caption{Density at $t=0.50$, $\varepsilon_0=10^{-6}$.}
    \label{fig:tb-density-050}
  \end{subfigure} \hfill
  \begin{subfigure}[t]{0.48\textwidth}
    \centering
    \includegraphics[width=\linewidth]{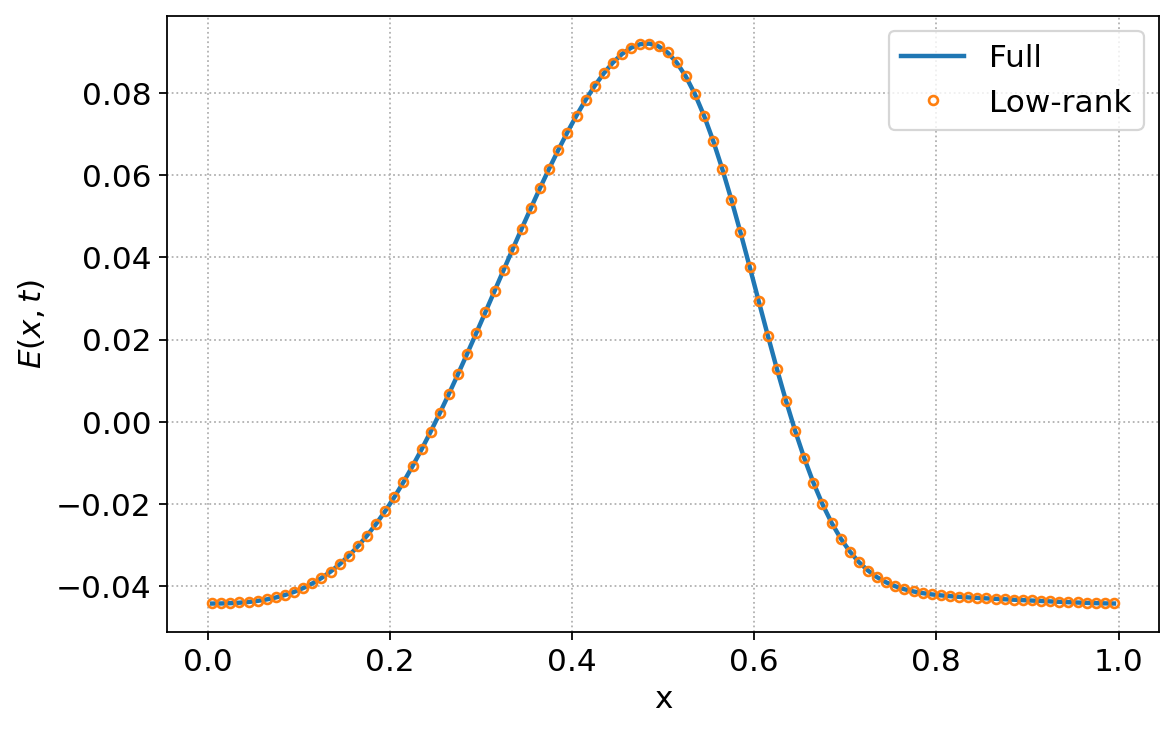}
    \caption{Electric field at $t=0.50$, $\varepsilon_0=10^{-6}$.}
    \label{fig:tb-efield-050}
  \end{subfigure}
  \caption{Bump-on-tail test: macroscopic density and electric field in the kinetic ($t=0.10$, $\varepsilon_0=1$) and fluid ($t=0.50$, $\varepsilon_0=10^{-6}$) regimes.}
  \label{fig:tb-density-efield}
\end{figure}

To visualize the phase-space dynamics, Figures~\ref{fig:tb-phases-010} and \ref{fig:tb-phases-050} show phase space plots of the full-tensor reference solution, the low-rank solution, and their absolute difference at representative times in the kinetic and fluid regimes, respectively. In the kinetic case ($t=0.10$, $\varepsilon_0=1$), the low-rank solution approximates the full-tensor reference solution well, accurately capturing the phase-space structures. In the fluid case ($t=0.50$, $\varepsilon_0=10^{-6}$), the phase-space distribution has nearly relaxed to a local Maxwellian, and the low-rank approximation is extremely accurate across the domain.

\begin{figure}[htbp]
  \centering
  \begin{subfigure}[t]{0.32\textwidth}
    \centering
    \includegraphics[width=\linewidth]{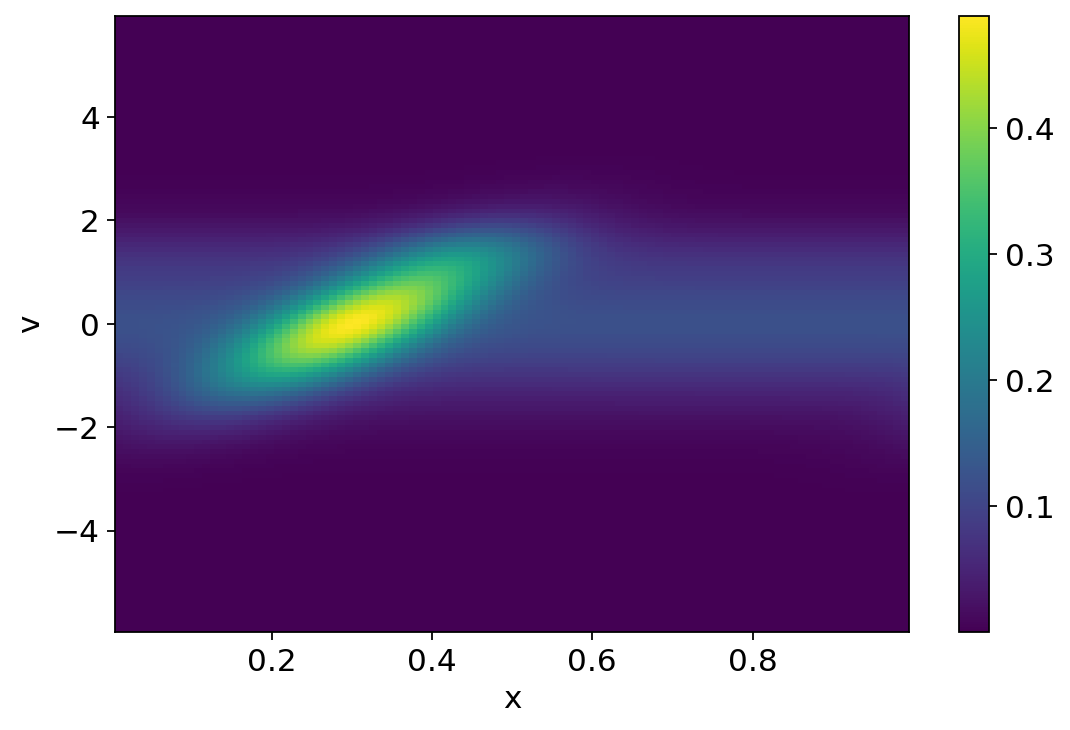}
    \caption{Full model ($t=0.10$).}
    \label{fig:tb-phase-full-010}
  \end{subfigure}\hfill
  \begin{subfigure}[t]{0.32\textwidth}
    \centering
    \includegraphics[width=\linewidth]{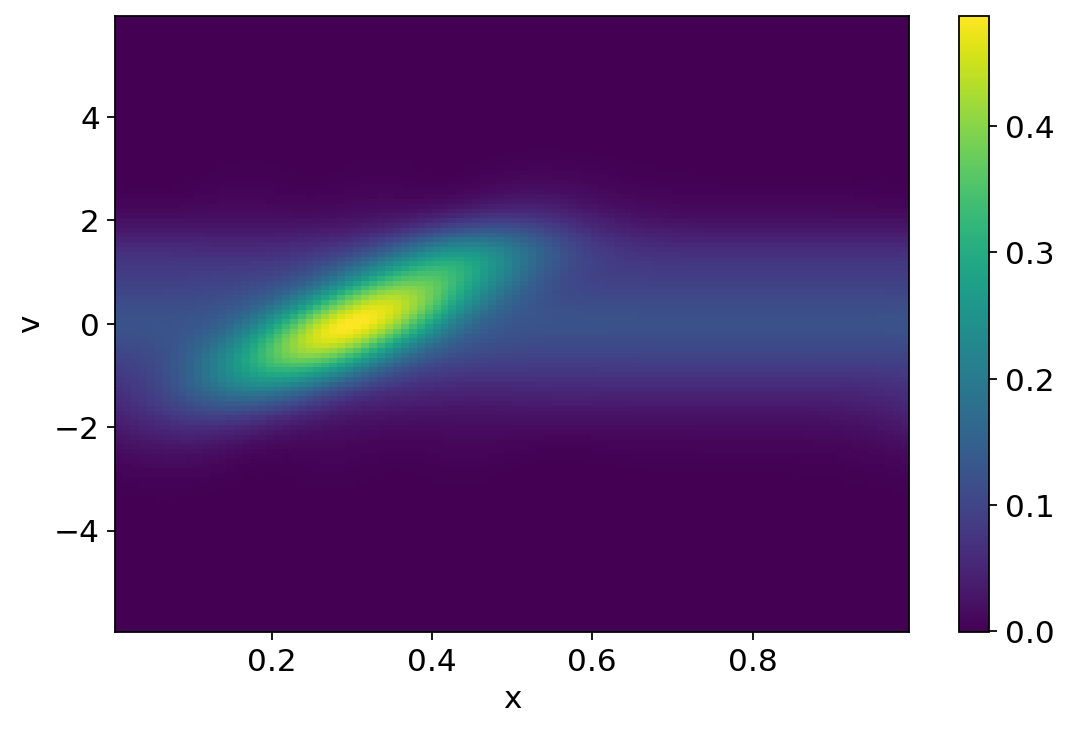}
    \caption{Low-rank model ($t=0.10$).}
    \label{fig:tb-phase-low-010}
  \end{subfigure}\hfill
  \begin{subfigure}[t]{0.32\textwidth}
    \centering
    \includegraphics[width=\linewidth]{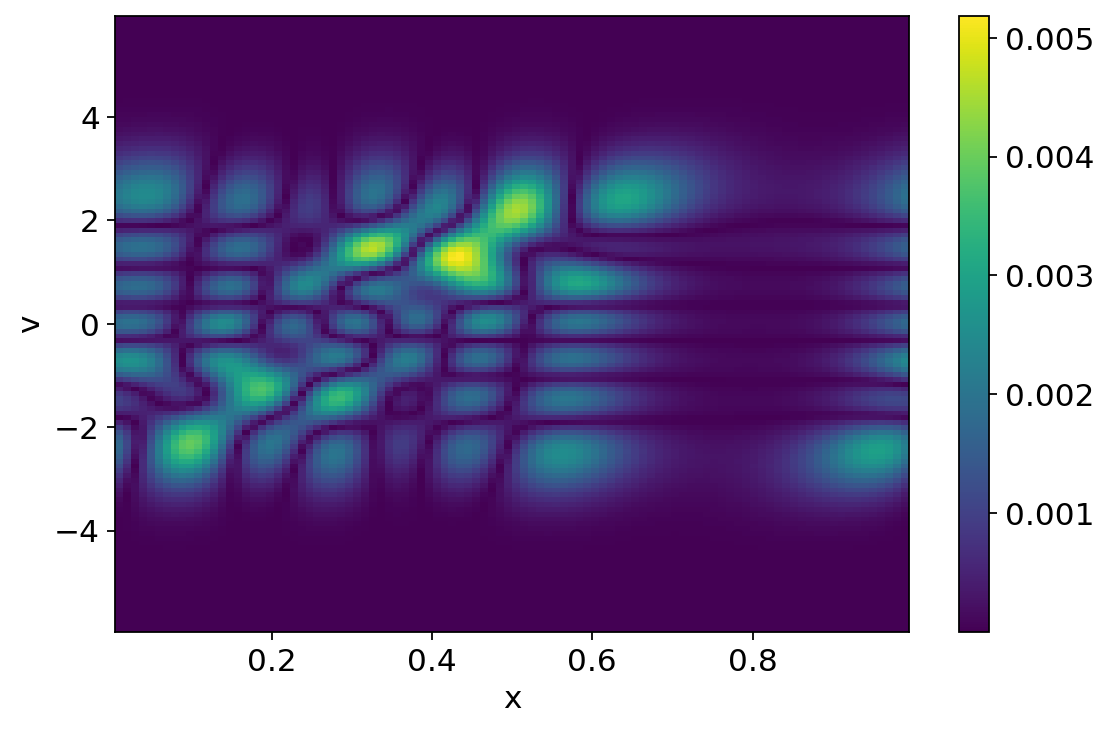}
    \caption{$|f_\text{full}-f_\text{low}|$ ($t=0.10$).}
    \label{fig:tb-phase-diff-010}
  \end{subfigure}
  \caption{Bump-on-tail phase space plots at $t=0.10$ in the kinetic regime ($\varepsilon_0=1$).}
  \label{fig:tb-phases-010}
\end{figure}

\begin{figure}[htbp]
  \centering
  \begin{subfigure}[t]{0.32\textwidth}
    \centering
    \includegraphics[width=\linewidth]{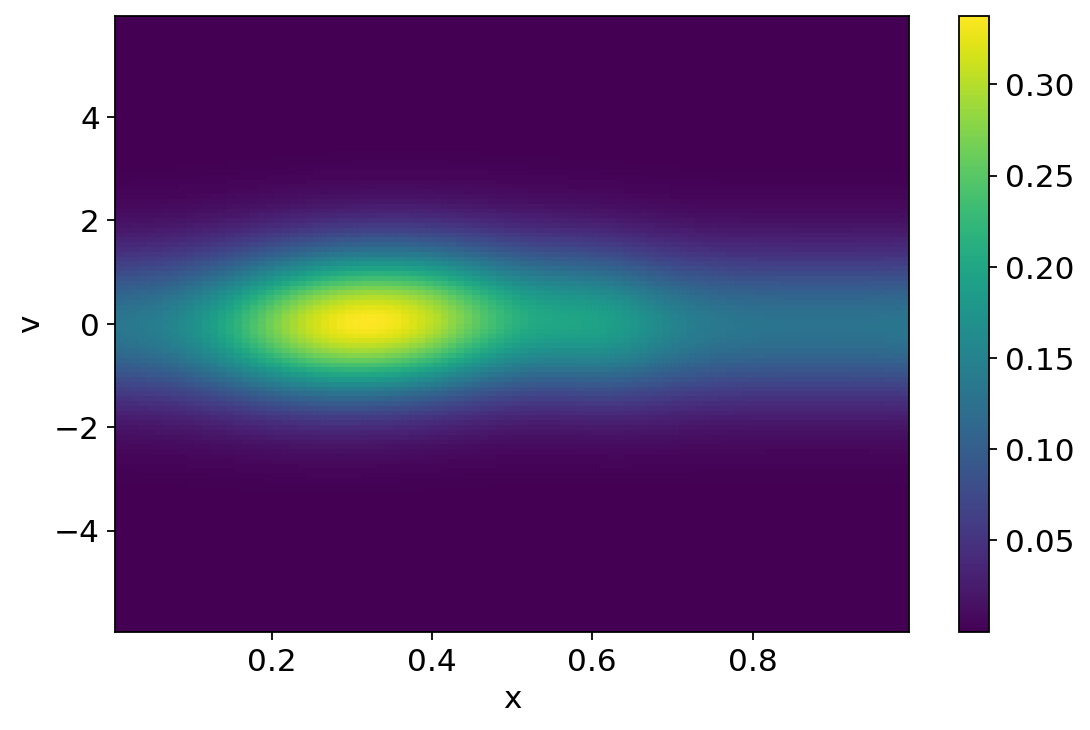}
    \caption{Full model ($t=0.50$).}
    \label{fig:tb-phase-full-050}
  \end{subfigure}\hfill
  \begin{subfigure}[t]{0.32\textwidth}
    \centering
    \includegraphics[width=\linewidth]{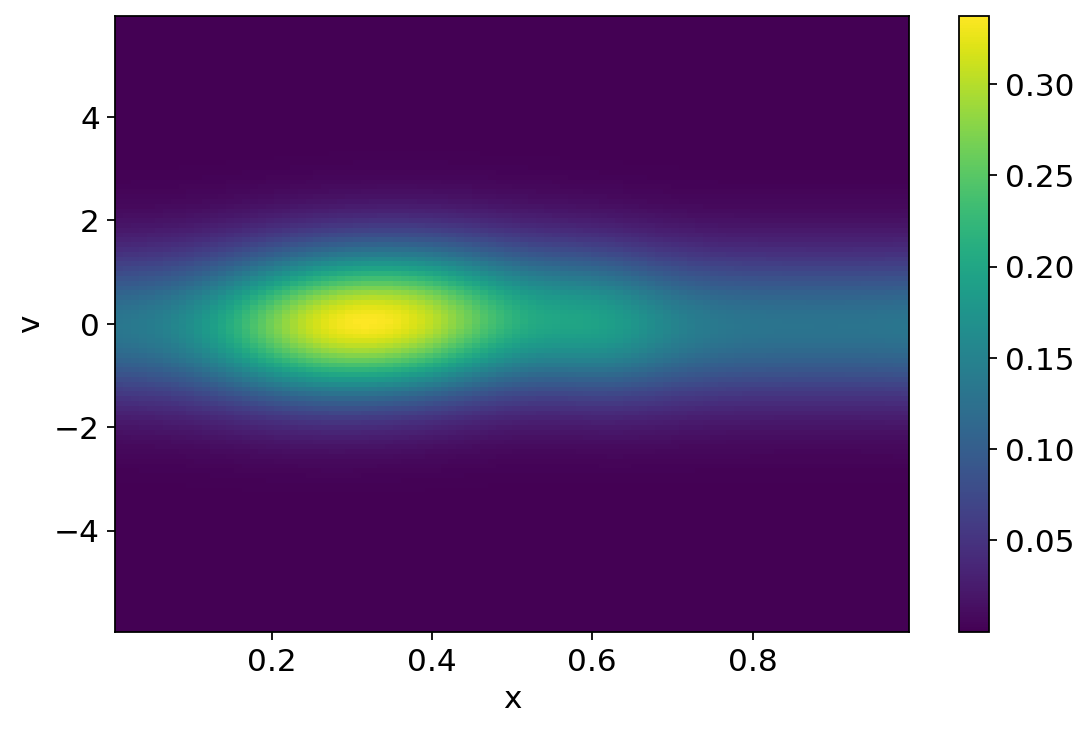}
    \caption{Low-rank model ($t=0.50$).}
    \label{fig:tb-phase-low-050}
  \end{subfigure}\hfill
  \begin{subfigure}[t]{0.32\textwidth}
    \centering
    \includegraphics[width=\linewidth]{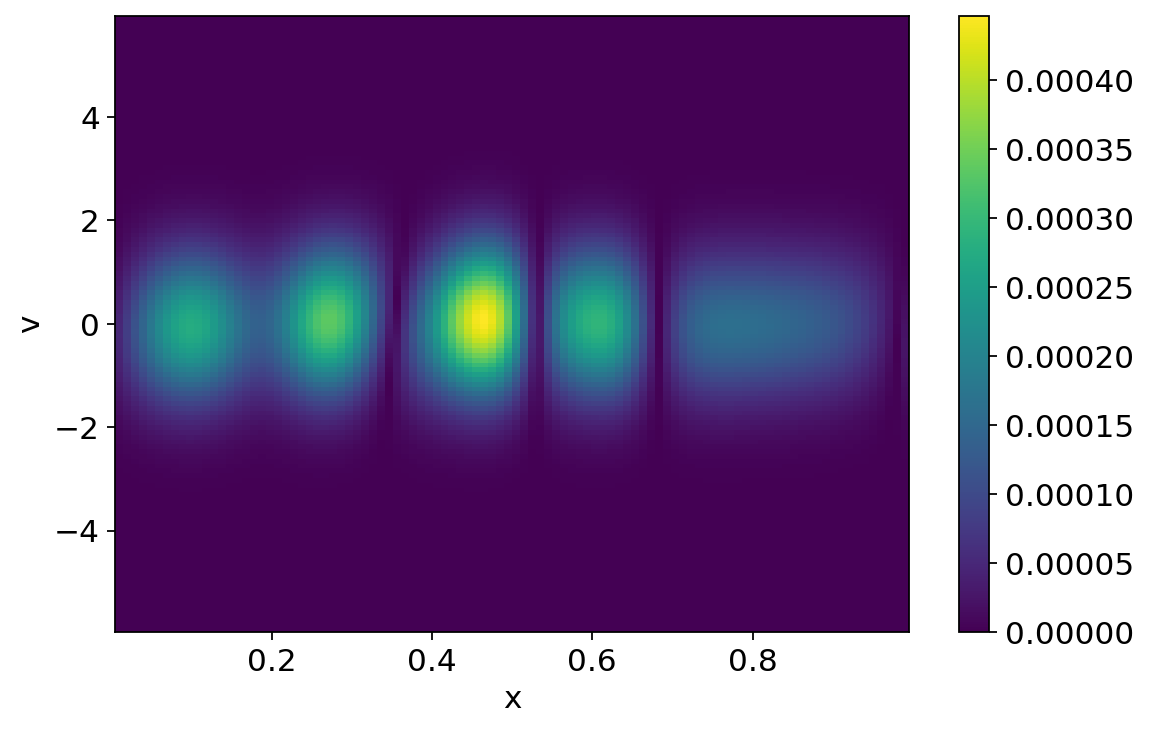}
    \caption{$|f_\text{full}-f_\text{low}|$ ($t=0.50$).}
    \label{fig:tb-phase-diff-050}
  \end{subfigure}
  \caption{Bump-on-tail phase space plots at $t=0.50$ in the fluid regime ($\varepsilon_0=10^{-6}$).}
  \label{fig:tb-phases-050}
\end{figure}

Finally, Figure~\ref{fig:tb-l1-errors} plots the global AP error $\mathcal E_{AP}(t)$ for both regimes. In the fluid case ($\varepsilon_0=10^{-6}$) the error decays rapidly and remains very small, indicating fast relaxation to the local Maxwellian shape. 
In the kinetic case ($\varepsilon_0=1$) the deviation remains larger but bounded, reflecting the persistent kinetic structures associated with the bump-on-tail instability. 

\begin{figure}[htbp]
  \centering
  \begin{subfigure}[t]{0.48\textwidth}
    \centering
    \includegraphics[width=\linewidth]{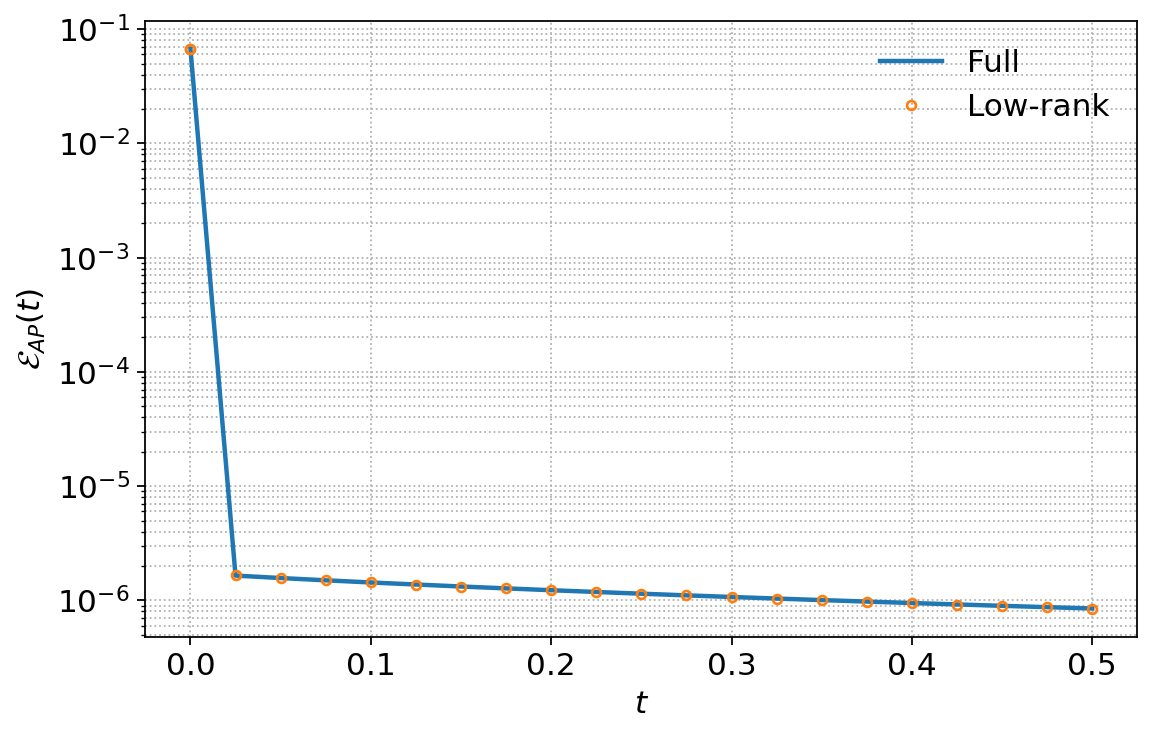}
    \caption{$\mathcal{E}_{AP}(t)$ vs.\ $t$ (fluid, $\varepsilon_0=10^{-6}$).}
    \label{fig:tb-l1-fluid}
  \end{subfigure}\hfill
  \begin{subfigure}[t]{0.48\textwidth}
    \centering
    \includegraphics[width=\linewidth]{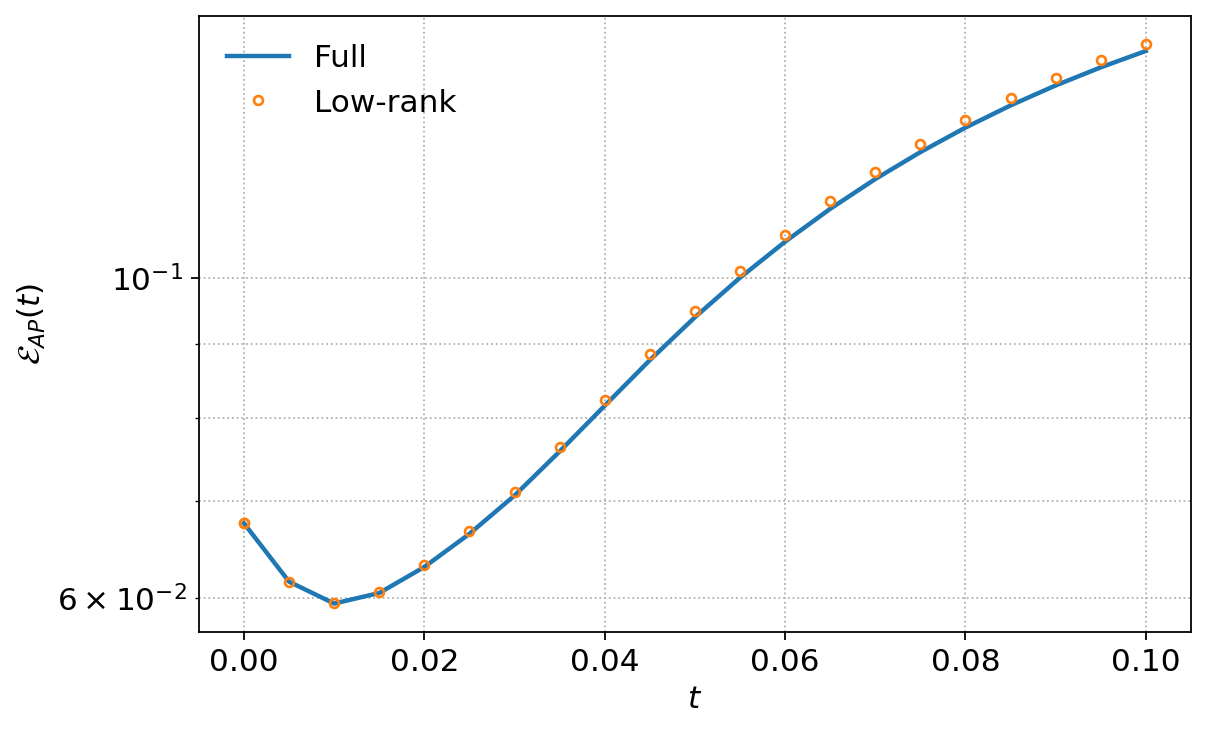}
    \caption{$\mathcal{E}_{AP}(t)$ vs.\ $t$ (kinetic, $\varepsilon_0=1$).}
    \label{fig:tb-l1-kinetic}
  \end{subfigure}
  \caption{Bump-on-tail AP error $\mathcal{E}_{AP}(t)$ over time for the first-order low-rank scheme in fluid and kinetic regimes.}
  \label{fig:tb-l1-errors}
\end{figure}

\section{Conclusions}
\label{sec:con}

We have presented a dynamical low-rank (DLR) method for the Vlasov--Poisson--Fokker--Planck system that effectively addresses both the curse of dimensionality and the stiffness of the high-field regime. The cornerstone of this method is a conservative, separable discretization of the Fokker--Planck operator. By constructing a three-point stencil where coefficients factor into space–only and velocity–only components, we enable the efficient implicit treatment of the Fokker--Planck operator directly on the low-rank manifold without requiring expensive full-tensor reconstruction. To handle the stiffness within the low-rank framework, we proposed both first-order and second-order schemes using proper IMEX treatment to ensure stability. We provided a theoretical analysis for the first-order low-rank scheme, proving an asymptotic–preserving property in the case of small spatial field fluctuations. Numerical benchmarks demonstrate that the proposed methods are asymptotically preserving and robust, maintaining the reduced computational and storage cost while accurately capturing macroscopic limits and kinetic structures at modest ranks. \modifyfirst{Extending the present separable AP IMEX construction to alternative DLR time integrators, such as BUG-type schemes~\cite{ceruti2022unconventional,CKL22} which avoid the backward-in-time S-step entirely and admit rank-adaptive variants, is a natural direction for future work.}

\section*{Acknowledgement}
This work was partially supported by DOE grant DE-SC0023164, NSF grants DMS-2409858 and IIS-2433957, and DoD MURI grant FA9550-24-1-0254.

\section*{Data Availability Statement}
The code reproducing the numerical experiments in this paper is publicly available in a GitHub repository: \url{https://github.com/shzhang3/DLR-VPFP}.

\bibliographystyle{siamplain}
\bibliography{references}
\end{document}